\documentclass[12pt]{article}
\usepackage{amsmath, amssymb, amsthm}
\usepackage{geometry}
\usepackage{graphicx}
\usepackage{hyperref}
\newtheorem{theorem}{Theorem}[section]
\newtheorem{lemma}[theorem]{Lemma}
\newtheorem{corollary}[theorem]{Corollary}
\theoremstyle{definition}
\newtheorem{definition}[theorem]{Definition}

\newcommand{\ddt}[1]{\frac{\mathrm{d}#1}{\mathrm{d}t}}

\newtheorem{remark}[theorem]{Remark}
\newtheorem{conj}[theorem]{Conjecture}

\title{Characterization of Infinite Ideal Polyhedra in Hyperbolic 3-Space via Combinatorial Ricci Flow}
\author{Huabin Ge, Bobo Hua, Hao Yu and Puchun Zhou}
\date{}

\begin{document}

\maketitle

\begin{abstract}

In his seminal work \cite{Ri96}, Rivin characterized finite ideal polyhedra in three-dimensional hyperbolic space. However, the characterization of infinite ideal polyhedra, as proposed by Rivin, has remained a long-standing open problem. In this paper, we introduce the combinatorial Ricci flow for infinite ideal circle patterns, a discrete analogue of Ricci flow on non-compact Riemannian manifolds, and prove a characterization of such circle patterns under certain combinatorial conditions.  Our results provide affirmative solutions to Rivin's problem.

\end{abstract}
\tableofcontents

\section{Introduction}
\setcounter{theorem}{-1}
The main objective of this paper is to investigate ideal polyhedra in hyperbolic 3-space within the framework of infinite combinatorial structures.

The problem of characterizing convex polyhedra is a classical one, with many interesting results. One of the most well-known is {Cauchy’s Rigidity Theorem}, which asserts that an Euclidean polytope with rigid faces is rigid, i.e., it cannot be deformed. For more modern developments, we refer the reader to the works of Alexandrov \cite{Alex05,Alex42}, Andreev \cite{An70A,An70B}, Bao-Banahon \cite{Bao}, Bobenko-Izmestiev \cite{BI06}, Bobenko-Pinkall-Springborn \cite{BPS2015}, Bobenko-Springborn \cite{BS04}, Bowers-Bowers-Pratt \cite{BBP}, Huang-Liu \cite{HL17}, Liu-Zhou \cite{Liu-Zhou}, Marden-Rodin \cite{MR90}, Rivin \cite{Ri94,Ri96,Ri03, Ri04}, Rivin-Hodgson \cite{RH93}, Roeder-Hubbard-Dunbar \cite{RHD07}, Rousset \cite{Rou04}, Schlenker \cite{Sch00,Sch05}, Schramm \cite{schramm2}, Springborn \cite{Springborn1,Springborn2} and Zhou \cite{Zhou23}.

The focus of this paper is on {ideal polyhedra}, which is a convex polyhedra inscribed in a sphere. The question of characterizing which polyhedra can be realized as convex polyhedra inscribed in a sphere dates back nearly 200 years. It was posed at the end of Jakob Steiner’s book \cite{Steiner}, though it can be traced further back to unpublished work by René Descartes around 1630 \cite{Fed}.

The first result was obtained by Steinitz in 1928 \cite{Steinitz}, who provided examples of polyhedra that cannot be realized with either an insphere or a circumsphere. Let $\mathcal{D} = (V, E, F)$ be a cellular decomposition of the sphere $\mathbb{S}^2$.   We denote by $\Theta_e$ the exterior dihedral angle of a polyhedron at the edge $e$. Subsequent progress remained scant until Rivin's landmark work \cite{Ri96}, where he proved the following theorem. 
\begin{theorem}[Rivin]\label{riv}
	Let $\mathcal{D} = (V, E, F)$ be a finite   cellular decomposition of the sphere.
    Suppose the angle $\Theta_e$, assigned to each edge $e$, satisfies the following:
	\begin{itemize}
		\item[(1)] $0 < \Theta_e < \pi$ for all $e \in E$;
		\item[(2)] $\sum_{e \in \gamma} \Theta_e = 2\pi$ for any loop $\gamma$ that is the coboundary of a single vertex;
		\item[(3)] $\sum_{e \in \gamma} \Theta_e > 2\pi$ for every edge cocycle $\gamma$ not bounding a single vertex.
	\end{itemize}
	Then there exists a unique (up to isometry) finite ideal hyperbolic polyhedron that is combinatorially equivalent to $\mathcal{D}$ and has exterior dihedral angles $\Theta_e$.
\end{theorem}

Here a \emph{cocycle} is a cycle in the Poincar\'e dual cellular decomposition, faces of which correspond to vertices of the original decomposition, vertices to faces, and edges to edges. The concept dual to
the boundary of a face is the \emph{coboundary} of a vertex.
Furthermore, Rivin also showed that these conditions are necessary, thus providing a complete characterization of finite ideal polyhedra in hyperbolic 3-space. He posed several open problems in \cite{Ri96}:  
(i) Extending Theorem \ref{riv} to {polyhedra with infinitely many vertices} (assuming all vertices have finite degree), which are called infinite ideal polyhedra (IIP) in this paper;  
(ii) Developing {efficient algorithms} to construct convex ideal polyhedra with prescribed dihedral angles.  

In this paper we consider infinite cellular decompositions of a surface. A cellular decomposition $\mathcal{D}=(V,E,F)$ is called \emph{locally finite} if every face in $F$ is combinatorially adjacent to only finitely many other faces in $F$. We say that $\mathcal{D}$ is \emph{connected} if for any pair of faces $f_1, f_2 \in F$, there exists a finite sequence of faces connecting $f_1$ to $f_2$, where each consecutive pair of faces shares a common edge. 
 Let $v_\infty \in V$ be a distinguished vertex, and let $F_\infty \subset F$ denote the set of faces incident to $v_\infty$. 
We write $f' \sim f$ if $f$ and $f'$ share an edge $e = [f', f] \in E$. Throughout this paper, we assume that any two faces share at most one edge, and any face has at least three boundary edges.  We denote by $\partial({F\backslash F_\infty})$ the set of faces belonging to $F\backslash F_\infty$ which are incident to a face in $F_\infty,$ and by $\mathrm{int}(F\backslash F_\infty)$ its complement in $F\backslash F_\infty$. For $ r \in (0, + \infty)^{ F \setminus F_\infty}$, a positive function on $F \setminus F_\infty$,
we define
\begin{align}
	\hat{K}_f &=
	\begin{cases}		\hspace{6.7em} 0,& \text{if } f \in \operatorname{int}(F \setminus F_\infty), \\[6pt]
		\sum\limits_{f' \in F_\infty, \,f \sim f'} 2\Theta_{[f', f]}, & \text{if } f \in \partial(F \setminus F_\infty),
	\end{cases} \label{K_euc_result} \\
	K_f(r) &= \sum_{f' \sim f} \alpha(f,f') \arcsin \left( \frac{r_{f'} \sin \Theta_{[f', f]}}{\sqrt{r_f^2 + r_{f'}^2 - 2\cos \Theta_{[f', f]} r_f r_{f'}}} \right), \label{K_euc_curvature_result}
\end{align}
with
\begin{equation}\label{K_euc}
	\alpha(f,f') =
	\begin{cases}
		1, & \text{if } f, f' \in \partial(F \setminus F_\infty), \\[6pt]
		2, & \text{otherwise}.
	\end{cases}
\end{equation}

In this paper, we provide an answer to Rivin's problem (i).
\begin{theorem}[Existence of IIP]\label{polyhedra_euc}
	Let $\mathcal{D} = (V, E, F)$ be a locally-finite connected infinite   cellular decomposition of the sphere, and suppose that an angle $\Theta_e \in (0, \pi)$ is assigned to each edge $e \in E$.  Define $\hat{K}$ and $K$ as in \eqref{K_euc_result} and \eqref{K_euc_curvature_result}. Assume the following hold:
	\begin{itemize}
		\item[(1)] There exists a universal constant $\delta > 0$ such that $\epsilon < \Theta_e < \pi - \delta$ for all $e \in E$;
		\item[(2)] $\sum_{e \in \gamma} \Theta_e = 2\pi$ for any loop $\gamma$ that is the coboundary of a single vertex;
		\item[(3)] $\sum_{e \in \gamma} \Theta_e \leq  (|\gamma|-2)\pi +  \hat{K}_f $ whenever $\gamma$ is the boundary of a face $f \in F \setminus F_\infty$;
		\item[(4)] there exists $r \in (0, + \infty)^{ F \setminus F_\infty}$ such that $K(r)\geq \hat K$.
	\end{itemize}
	Then there exists an infinite ideal hyperbolic polyhedron that is combinatorially equivalent to $\mathcal{D}$ and has exterior dihedral angles $\Theta_e$.
\end{theorem}

We further introduce a special type of infinite ideal hyperbolic polyhedron, referred to as a \emph{half-infinite ideal polyhedron (HIIP)}. By definition, an infinite ideal polyhedron $P = (V, E, F)$ is called \emph{half-infinite} if there exists a unique face $f_\infty$ such that  $\partial f_\infty$ is the subset of the set of limit points of $V \subset \mathbb{S}^2$. Note that there is no finite analog for a half-infinite ideal polyhedron. In the Poincaré ball model, such a polyhedron necessarily contains a sequence of infinitesimally small faces accumulating toward a point $f_{\infty}$ on the boundary.

\begin{theorem}[Existence of HIIP]\label{polyhedra_hyp}
	Let $\mathcal{D} = (V, E, F)$ be an infinite   cellular decomposition of the sphere. There exists a unique face $f_\infty \in F$ such that $\partial f_\infty$ is the subset of the set of limit points of $V \subset \mathbb{S}^2$. Let $\mathcal{D}$ be locally finite and connected except $f_\infty$, and $\Theta_e \in (0, \pi)$ be assigned to each edge $e$. Suppose there exists a constant ${c} > 0$ such that
	\begin{itemize}
		\item[(1)] $0 < \Theta_e < \pi$ for all $e \in E$;
		\item[(2)] $\sum_{e \in \gamma} \Theta_e = 2\pi$ for any loop $\gamma$ that is the coboundary of a single vertex;
		\item[(3)] $\sum_{e \in \gamma} \Theta_e \leq (|\gamma|-2)\pi + |\gamma| {c}$ whenever $\gamma$ is the boundary of a face.
	\end{itemize}
	Then there exists a half-infinite ideal hyperbolic polyhedron combinatorially equivalent to $\mathcal{D}$ with exterior dihedral angles $\Theta_e$.
\end{theorem}
Moreover, we have the following conjecture.
\begin{conj}
    Condition (3) in Theorems~\ref{polyhedra_euc} and~\ref{polyhedra_hyp} can be further relaxed such that it only needs to hold outside a finite sub-complex of $\mathcal{D}$.
\end{conj}
This conjecture is motivated by the fact that the transience (resp. recurrence) of a graph remains unchanged when a finite subgraph is modified. In fact, the works of He–Schramm~\cite{He2} and Oh~\cite{Oh} also provide supporting evidence for this conjecture. The conditions presented in Theorems \ref{polyhedra_euc} and \ref{polyhedra_hyp} have the advantage of being local, seperating, and readily verifiable, whereas those in Rivin's theorem are intricately linked to the underlying combinatorial structure, rendering them more challenging to verify.

Our approach for proving Theorems \ref{polyhedra_euc} and \ref{polyhedra_hyp} relies on the theory of {circle patterns}, a method introduced by Thurston \cite{Th76}, who observed a deep connection between Andreev’s theorem \cite{An70A, An70B} and Koebe’s circle packing theorem \cite{Ko36}. Thurston further conjectured that conformal maps could be approximated by graph isomorphisms between circle packings, a result later proved by Rodin-Sullivan \cite{ro2}. Consequently, circle patterns are regarded as a discrete analog of conformal mappings and a powerful tool in 3-dimensional geometric topology \cite{st}. 
The existence of IIP and HIIP can be interpreted as the existence of infinite ideal circle patterns. A key technique in studying circle patterns is the use of {differential equations}. In 2002, inspired by Hamilton's work on Ricci flow \cite{H82, H98}, Chow-Luo \cite{CL03} introduced \emph{combinatorial Ricci flow} as a new method for analyzing circle patterns, formulated as
\begin{equation}\label{flow}
	\frac{\mathrm{d} r_i}{\mathrm{d} t} = -K_i \cdot s(r_i), \quad \forall v_i \in V,
\end{equation}
where $K_i$ is the \emph{discrete Gaussian curvature} at the vertex $v_i$, and $s(r_i) = r_i$ in the Euclidean background and $s(r_i) = \sinh r_i$ in the hyperbolic case. They showed that for surfaces of positive genus, the solution converges exponentially to a Thurston's circle pattern.

The first, second authors, and Zhou \cite{GHZ21} onstudied the combinatorial Ricci flow for {ideal circle patterns}, focusing on finite cellular decompositions. However, to establish Theorems \ref{polyhedra_euc} and \ref{polyhedra_hyp}, we must extend their theory to the infinite setting. Since many methods, including combinatorial and topological ones, fail in the infinite setting, the problem becomes particularly challenging. Moreover, the combinatorial Ricci potential,  the key energy functional for the problem in the finite case, becomes infinite for infinite circle packings. As a result, the variational methods developed in \cite{Ve91} are no longer applicable to infinite combinatorial Ricci flows. This necessitates the development of new ideas and methods suitable for our setting. 

Very recently, the first, second, and fourth authors first introduced the infinite combinatorial Ricci flows \cite{Ge-Hua-Zhou-infinite}, extending results of Chow-Luo \cite{CL03} and He \cite{He}. After that, Li, Lu and the third author \cite{Li-Lu-Yu}, Zhao and the fourth author \cite{Zhao-Zhou} introduced two infinite flows for total geodesic curvatures in spherical and hyperbolic background geometries, respectively.

Let $\mathcal{D} = (V, E, F)$ be an infinite cellular decomposition on a non-compact surface $S$, and let $\Theta$ be a function on $E$. Denote by $d_i$ the degree of a vertex $v_i \in V$, i.e., the number of edges connected to $v_i$. Let $G$ be the 1-dimensional skeleton of $\mathcal{D}$. For any $v_i, v_j \in V$, the combinatorial distance $d(v_i, v_j)$ is defined as the graph distance in $G$. Throughout this paper, we assume $d(v_i, v_j) < +\infty$ for all $v_i, v_j \in V$. For an ideal circle pattern, we need a condition:

\begin{itemize}
	\item[(C1)] If \( e_1 + \cdots + e_m \) forms the boundary of a face \( f \), then \( \sum_{l=1}^{m} \Theta(e_l) = (m - 2) \pi \).
\end{itemize}

We first establish the long-time existence of the flow \eqref{flow}.

\begin{theorem}[Long-Time Existence]\label{existence}
	In both hyperbolic and Euclidean background geometries, if $\Theta \in (0, \pi)^E$ and satisfies (C1), then for any initial value $r(0)$, there exists a solution $r(t)$ to equation \eqref{flow} that exists for all time $t \geq 0$.
\end{theorem}

The proof of Theorem \ref{existence} is inspired by the work of Shi \cite{shi}. Under additional conditions, we can also guarantee the uniqueness of the solution.

\begin{theorem}[Uniqueness I]
Let $R$ denote the set of all functions $r\in (0,\infty)^{V\times [0,\infty)}$ whose discrete Gaussian curvatures are uniformly bounded on $V \times [0, \infty)$.
	In both hyperbolic and Euclidean background geometries, let $\delta > 0$ be an arbitrarily small constant.  If $\Theta \in (\delta, \pi - \delta)^E$ satisfies \textup{(C1)}, then for any initial value $r(0)$ with bounded discrete Gaussian curvature,  equation \eqref{flow} admits a unique solution $r\in R$.
\end{theorem}

\begin{theorem}[Uniqueness II]\label{unique_for_uni_degree}
	In both hyperbolic and Euclidean background geometries, let $\delta$ be an arbitrarily small positive constant. If $\Theta \in (\delta, \pi - \delta)^E$ satisfies (C1), and there exists a positive constant $\hat{d}$ such that $d_i \leq \hat{d}$ for all $v_i \in V$, then for any initial value $r(0)$, the solution $r(t)$ to equation \eqref{flow} is unique.
\end{theorem}

Since convergence of the Ricci flow typically implies the existence of a \emph{good ideal circle pattern}—namely, a circle packing metric with zero discrete Gaussian curvature—it is important to analyze the behavior of $r(t)$ as $t \to +\infty$.

In hyperbolic background geometry, we have the following result.

\begin{theorem}[Convergence I]\label{convergence1}
	In hyperbolic background geometry, if $\Theta \in (0, \pi)^E$ satisfies (C1) and the initial curvature satisfies $K_i(r(0), \Theta) \leq 0$ for all $v_i \in V$, then the solution $r(t)$ to equation \eqref{flow} converges to a good ideal circle pattern.
\end{theorem}
If there exists an initial value $r(0)$ with non-positive discrete Gaussian curvature, by the above theorem, the cellular decomposition $\mathcal{D}$ admits a good ideal circle pattern.
This gives a criterion for the existence of a good ideal circle pattern, and provides an algorithm for finding the desired ideal circle pattern using the combinatorial Ricci flow.  This answers Rivin's problem (ii).

In Euclidean background geometry, consider the infinite cellular decomposition \( \mathcal{D} = (V, E, F) \) of the plane $\mathbb{R}^2$, whose 1-dimensional skeleton $G = (V, E)$ corresponds to a lattice structure as shown in Figure~\ref{fig:1}. 
\begin{figure}[h]
	\centering
	\includegraphics[width=0.5\textwidth]{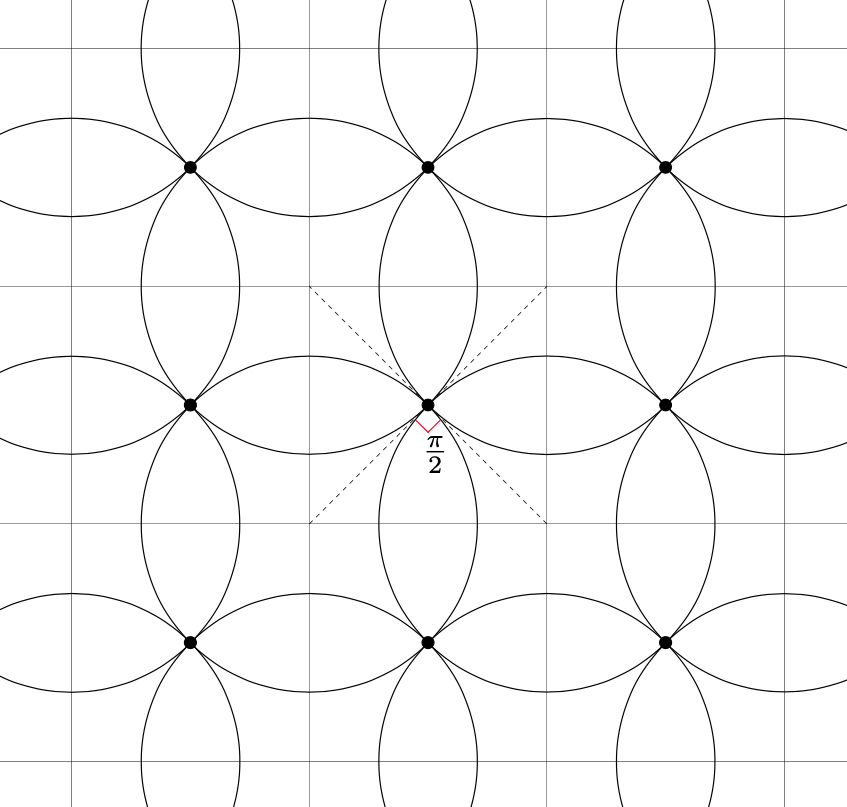}
	\caption{An infinite ideal circle pattern with $\Theta_e = \frac{\pi}{2}$ in $\mathbb{R}^2$}
	\label{fig:1}
\end{figure}
\begin{theorem}[Convergence II]\label{convergence2}
	In Euclidean background geometry, for the cellular decomposition \( \mathcal{D} = (V, E, F) \) shown in Figure~\ref{fig:1}, if $\Theta_e = \frac{\pi}{2}$ for all $e \in E$, then there exists a universal constant $\epsilon > 0$ such that for any initial data $r(0)$ satisfying $\|\ln r(0)\|_{l^2} \leq \epsilon$, the solution $r(t)$ to equation \eqref{flow} converges to the regular circle packing metric
	\[
	r_{\text{reg}} \equiv 1.
	\]
	Moreover, the logarithmic variable $u(t) = \ln r(t)$ satisfies $u \in C^1([0, \infty), l^2(V))$,
	\[
	\|u(t)\|_{l^\infty} \to 0 \quad \text{as } t \to \infty,
	\]
	and
	\[
	\mathcal{E}(u(t)) := \sum_{[i,j] \in E} |u(i,t) - u(j,t)|^2 \leq C(1 + t)^{-1},
	\]
	where $C = C(\epsilon)$.
\end{theorem}
\begin{remark}
    The infinite ideal circle pattern we obtained in the above theorem is the regular $SG$ pattern, which was systematically studied by Schramm \cite{schramm}.
\end{remark}

In the following, we want to find some local combinatorial conditions to guarantee the existence of an initial metric with non-positive curvatures, yielding a good ideal circle pattern.
To address this problem, following the work of the first author and Lin \cite{GL24} and Li-Lin-Shen \cite{Li-Lin-Shen}, we consider the \emph{character} $\mathcal{L}_i(\mathcal{D}, \Theta)$ for an infinite decomposition at each vertex \(v_i \in V\)
\[
\mathcal{L}_i(\mathcal{D}, \Theta) = \sum_{j \sim i} \Theta_{ij},
\]
where the sum runs over all vertices \(v_j\) adjacent to \(v_i\). In this paper, we further introduce the \emph{average normalized character} $\bar{\mathcal{L}}_i(\mathcal{D}, \Theta)$ 
\[
\bar{\mathcal{L}}_i(\mathcal{D}, \Theta) = \frac{\mathcal{L}_i(\mathcal{D}, \Theta) - 2\pi}{d_i}.
\]

\begin{theorem}\label{ch_hyper_exist}
	In hyperbolic background geometry, if $\Theta \in (0, \pi)^E$ satisfies (C1), and there exists a constant $\hat{c} > 0$ such that for all $v_i \in V$,
	\[
	\bar{\mathcal{L}}_i(\mathcal{D}, \Theta) \geq \hat{c},
	\]
	then there exists a good ideal circle pattern.
\end{theorem}

In the proof of Theorem~\ref{ch_hyper_exist}, we construct an initial metric $r(0)$ such that $K_i(r(0), \Theta) \leq 0$. By the combinatorial Ricci flow we find a good ideal circle pattern, however, the convergence rate is generally unknown.
 By imposing some natural conditions, one can ensure the exponential convergence of the flow, yielding an effective algorithm for the construction of such patterns.

\begin{theorem}\label{ch_hyper}
	In hyperbolic background geometry, let $\delta > 0$ be arbitrarily small. If $\Theta \in (\delta, \pi - \delta)^E$ satisfies (C1), and there exist positive constants $\hat{c}, \hat{d}$ such that for all $v_i \in V$,
	\begin{itemize}
		\item[(1)] $\bar{\mathcal{L}}_i(\mathcal{D}, \Theta) \geq \hat{c}$,
		\item[(2)] $d_i \leq \hat{d}$,
	\end{itemize}
	then for any two constants $\hat{r}_1, \hat{r}_2 > 0$, if the initial data $r(0)$ satisfies $\hat{r}_1 \leq r_i(0) \leq \hat{r}_2$ for all $v_i \in V$, then the solution $r(t)$ to equation \eqref{flow} is unique and converges exponentially to a good ideal circle pattern.
\end{theorem}

In the Euclidean case, the character function also yields structural information.

\begin{theorem}\label{ch_euclid}
	In Euclidean background geometry, let $\delta > 0$ be arbitrarily small. If $\Theta \in (\delta, \pi - \delta)^E$ satisfies (C1), then there exists a good ideal circle pattern if
	\begin{itemize}
		\item[(1)] $\bar{\mathcal{L}}_i(\mathcal{D}, \Theta) \geq 0$ for all $v_i \in V$;
		\item[(2)] there exists a $\mathcal{D}$-type circle packing metric $\tilde{r}$ with non-negative discrete Gaussian curvature.
	\end{itemize}
	Moreover, for any initial data $r(0)$ such that
	\begin{itemize}
		\item[(i)] $r_i(0) \geq \hat{r}_1 > 0$ for all $v_i \in V$;
		\item[(ii)] $K_i(r(0)) \geq 0$ for all $v_i \in V$,
	\end{itemize}
	there exists a solution $r(t)$ to equation \eqref{flow} that converges to a good ideal circle pattern.
\end{theorem}

In the forthcoming work \cite{GYZ}, we will consider
the rigidity and uniformization for infinite ideal circle patterns, where we generalize He \cite{He} and He-Schramm’s work \cite{He2} on infinite ideal circle patterns.
We say that a cellular decomposition $\mathcal{D}$ with intersection angles $\Theta \in (0,\pi)^E$ is \emph{ICP-parabolic} if there exists a locally finite, embedded ideal circle pattern of $\mathcal{D}$ in $\mathbb{R}^2$ with intersection angles $\Theta$. Similarly, $(\mathcal{D}, \Theta)$ is called \emph{ICP-hyperbolic} if there exists a locally finite, embedded ideal circle pattern in the unit disk with intersection angles $\Theta$. Motivated by the discussion in~\cite{Ge23}, we formulate the following conjecture.
\begin{conj}
\label{con-ge}
    Let $\mathcal{D}=(V,E,F)$ be a cellular decomposition of the sphere with dihedral angle $\Theta \in\left(0, \pi\right)^E$, then the following statements are equivalent.
    \begin{itemize}
        \item[(A)] There exists IIP (resp. HIIP) with dihedral angle $\Theta$.
        \item[(B)] The 1-skeleton of $\mathcal{D}$ is VEL-parabolic (resp. VEL-hyperbolic).
        \item[(C)] The 1-skeleton of $\mathcal{D}$ is recurrent (resp. transient).
        \item[(D)] There exists a locally finite, embedded ideal circle pattern of $\mathcal{D}$ in $\mathbb{R}^2$ (resp. unit disk) with intersection angle $\Theta$.
        \item[(E)] The corresponding Ricci flow for ideal circle patterns in Euclidean background geometry (resp. hyperbolic background geometry) converges.
    \end{itemize}
\end{conj}
In~\cite{GYZ}, we will establish the equivalence between the notions of VEL-parabolic (resp. VEL-hyperbolic), VEL-parabolic (resp. VEL-hyperbolic), recurrent (resp. transient), ICP-parabolic (resp. ICP-hyperbolic), and IIP (resp. HIIP). Together with the results of this paper, this provides a satisfactory solution to Conjecture \ref{con-ge}.

The organization of the paper is as follows. In Section \ref{Sec:2}, we introduce some basic concepts of ideal circle patterns. Moreover, we study maximum principles for parabolic equations on infinite graphs. In Section \ref{Sec:3} we will introduce the combinatorial Ricci flows for infinite ideal circle patterns, and establish the well-posedness and convergence results. Finally, in Section \ref{ideal_polyhe}, we apply the existence results of infinite ideal circle patterns to prove the existence of IIP and HIIP.

\section{Definitions and preliminaries}\label{Sec:2}

\subsection{Ideal circle pattern and vertex curvature}

Let \( S \) be a non-compact surface, and let \( \mathcal{D} = (V, E, F) \) be an infinite cellular decomposition of \( S \), where \( V \), \( E \), and \( F \) denote the sets of vertices, edges, and faces, respectively. Given a metric \( \mu \) on \( S \), we denote the surface as \( (S, \mu) \).

A \emph{circle pattern} \( \mathcal{P} \) on \( (S, \mu) \) is a collection of disks \( \{D_v\}_{v \in V} \). It is called a \emph{\( \mathcal{D} \)-type circle pattern} if there exists a geodesic cellular decomposition \( \mathcal{D}(\mu) \) such that:
\begin{enumerate}
	\item \( \mathcal{D}(\mu) \) is isotopic to \( \mathcal{D} \);
	\item the vertex set of \( \mathcal{D}(\mu) \) coincides with the centers of the disks in \( \mathcal{P} \).
\end{enumerate}
\begin{figure}[h]
	\centering
	\includegraphics[width=0.5\textwidth]{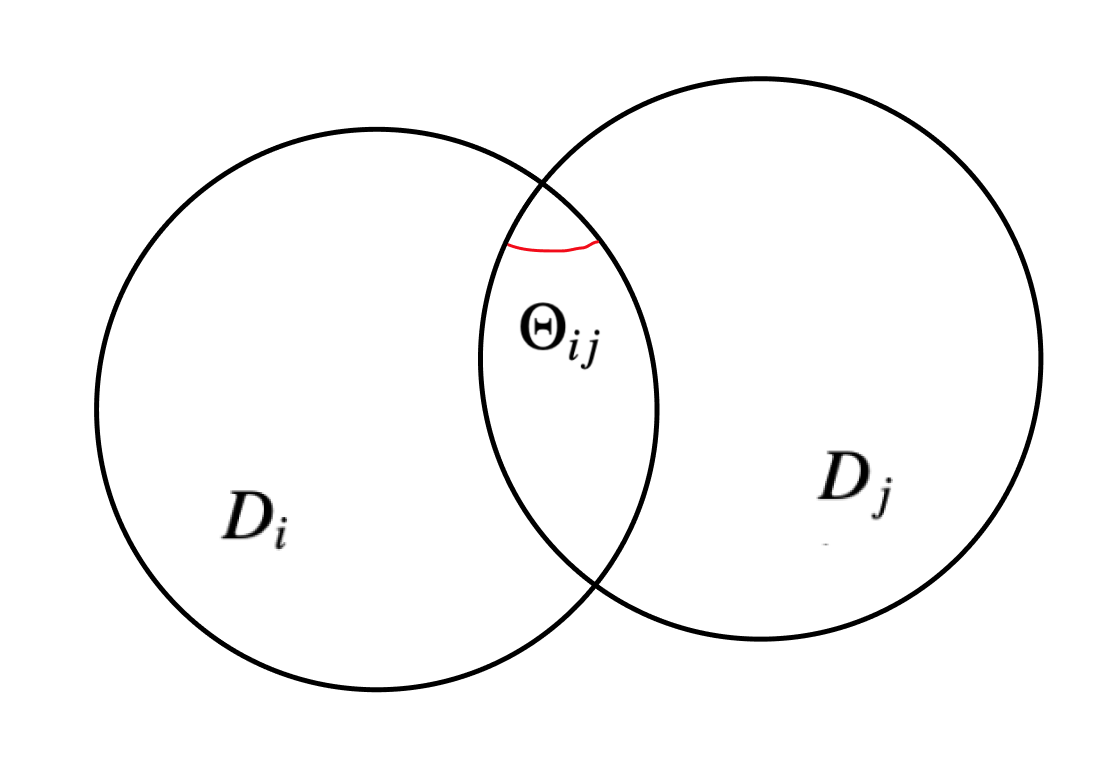}
	\caption{The intersection angle $\Theta_{ij}$ between disks $D_i$ and $D_j$}
	\label{fig:angle}
\end{figure}
The angle \( \Theta_{ij} \), as shown in Figure~\ref{fig:angle}, is called the \emph{intersection angle} between disks \( D_i \) and \( D_j \). A circle pattern \( \mathcal{P} \) is called a \emph{\( \mathcal{D} \)-effective circle pattern} if for every edge \( [v_i, v_j] \in E \), the corresponding disks \( D_i \) and \( D_j \) intersect. For such patterns, we define the \emph{intersection angle function}:
\[
\Theta: E \to [0, \pi), \quad \Theta([v_i, v_j]) = \Theta_{ij}.
\]
In this paper, all circle patterns are assumed to be \( \mathcal{D} \)-effective.

Let \( \mathbb{D}_v \) denote the interior of disk \( D_v \). An \emph{interstice} is a connected component of \( S \setminus \bigcup_{v \in V} \mathbb{D}_v \).

\begin{definition}
	A \( \mathcal{D} \)-type circle pattern \( \mathcal{P} \) is called an \emph{ideal circle pattern} if:
	\begin{itemize}
		\item[(i)] There is a one-to-one correspondence between the interstices and the faces in \( \mathcal{D}(\mu) \).
		\item[(ii)] Each interstice contains exactly one point.
	\end{itemize}
\end{definition}

For an ideal circle pattern, the following conditions hold:
\begin{itemize}
	\item[(C1)] If \( e_1 + \cdots + e_m \) bounds a face \( f \), then \( \sum_{l=1}^{m} \Theta(e_l) = (m - 2)\pi \).
	\item[(C2)] \( \Theta(e) \in (0, \pi) \) for all \( e \in E \).
\end{itemize}

We now describe Thurston’s construction for ideal circle patterns. For each face \( f \in F \), select a point \( v_f \in f \) and connect it to each vertex of \( f \) to form a triangulation \( \mathcal{T}(\mathcal{D}) = (V^*, E^*, F^*) \), where:
\[
V^* = V \cup F, \quad E^* = E \cup \{[v, v_f] \mid v \text{ is a vertex of } f\},
\]
\[
F^* = \{[e, v_f] \mid e \text{ is an edge of } f\}.
\]
The elements in \( V^* \setminus V \) are called \emph{star vertices}, and those in \( V \) are called \emph{primal vertices}. 

Let \( r: V \to \mathbb{R}_+ \) be a \emph{circle packing metric}, and let \( \Theta: E \to (0, \pi) \) be an intersection angle function. We write \( r_i = r(v_i) \) and \( \Theta_{ij} = \Theta([v_i, v_j]) \).

\begin{figure}[h]
	\centering
	\includegraphics[width=0.5\textwidth]{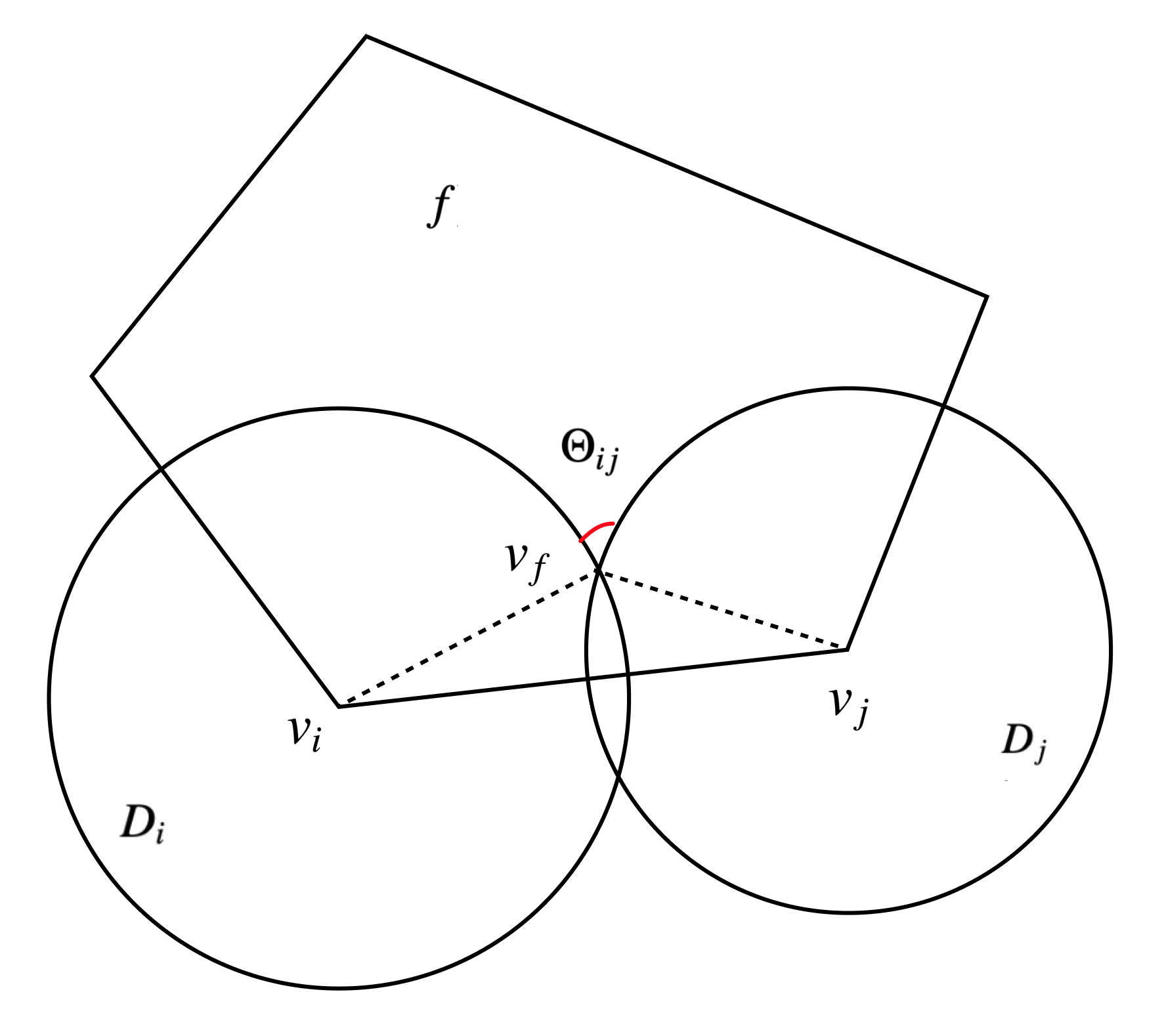}
	\caption{Triangle $\triangle(v_i v_j v_f)$ within a face $f$}
	\label{fig:thurstonconstructionidealcirclepatterntriangle}
\end{figure}
As illustrated in Figure~\ref{fig:thurstonconstructionidealcirclepatterntriangle}, for a face \( f \in F \), select an edge \( \overline{v_i v_j} \subset \partial f \) and the point \( v_f \in f \). This defines the triangle \( \triangle(v_i v_j v_f) \). Let \( l_{ij}, l_{jf}, l_{fi} \) denote the side lengths of \( \overline{v_i v_j}, \overline{v_j v_f}, \) and \( \overline{v_i v_f} \), respectively, where:
\[
l_{jf} = r_i, \quad l_{fi} = r_j.
\]
In hyperbolic geometry:
\[
l_{ij} = \cosh^{-1}(\cosh r_i \cosh r_j + \sinh r_i \sinh r_j \cos \Theta_{ij}).
\]
In Euclidean geometry:
\[
l_{ij} = \sqrt{r_i^2 + r_j^2 + 2 r_i r_j \cos \Theta_{ij}}.
\]

By gluing all such triangles in \( \mathcal{T}(\mathcal{D}) \) along common edges, we obtain a cone metric \( \mu(r, \Theta) \) on \( S \). At each star vertex \( v_f \), if \( e_1, \ldots, e_m \) are the edges of face \( f \), then by (C1), the total angle around \( v_f \) is:
\[
\sum_{i=1}^{m} (\pi - \Theta(e_i)) = m\pi - \sum_{i=1}^{m} \Theta(e_i) = m\pi - (m - 2)\pi = 2\pi,
\]
so \( \mu(r, \Theta) \) is smooth at star vertices. However, it may have singularities at primal vertices.

Let \( \sigma(v_i) \) be the total angle around a primal vertex \( v_i \). Define the \emph{vertex curvature} at \( v_i \) as:
\begin{equation}\label{vertex_curvature}
	K_i = 2\pi - \sigma(v_i).
\end{equation}
This defines a curvature function:
\[
\begin{aligned}
	K: V & \to \mathbb{R} \\
	v_i & \mapsto K_i.
\end{aligned}
\]
Thurston’s construction thus yields a curvature map:
\[
K(r, \Theta): \mathbb{R}_+^V \times (0, \pi)^E \to \mathbb{R}^V.
\]
For simplicity, we do not distinguish between \( K \) and \( K(r, \Theta) \). A circle packing metric with zero discrete Gaussian curvature is called a \emph{good ideal circle pattern}.

\subsection{Geometry of two-circle configurations}

In this section, we give some results regarding ideal circle patterns. As shown in Figure~\ref{fig:twocircleconf}, we denote by $\triangle(v_i v_j v_f)$ the triangle formed by two centers $v_i, v_j$ and one point $v_f$, which is an intersection point of the circles $D_i, D_j$. Let $\theta_{ij}, \theta_{ji}$ be the corresponding inner angles at the centers, $l_{ij}$ the length of $\overline{v_i v_j}$, and $d_{ij}$ the length of the altitude to the side $\overline{v_i v_j}$ of $\triangle(v_i v_j v_f)$.

\begin{figure}[htbp]
	\centering
	\includegraphics[width=0.8\textwidth]{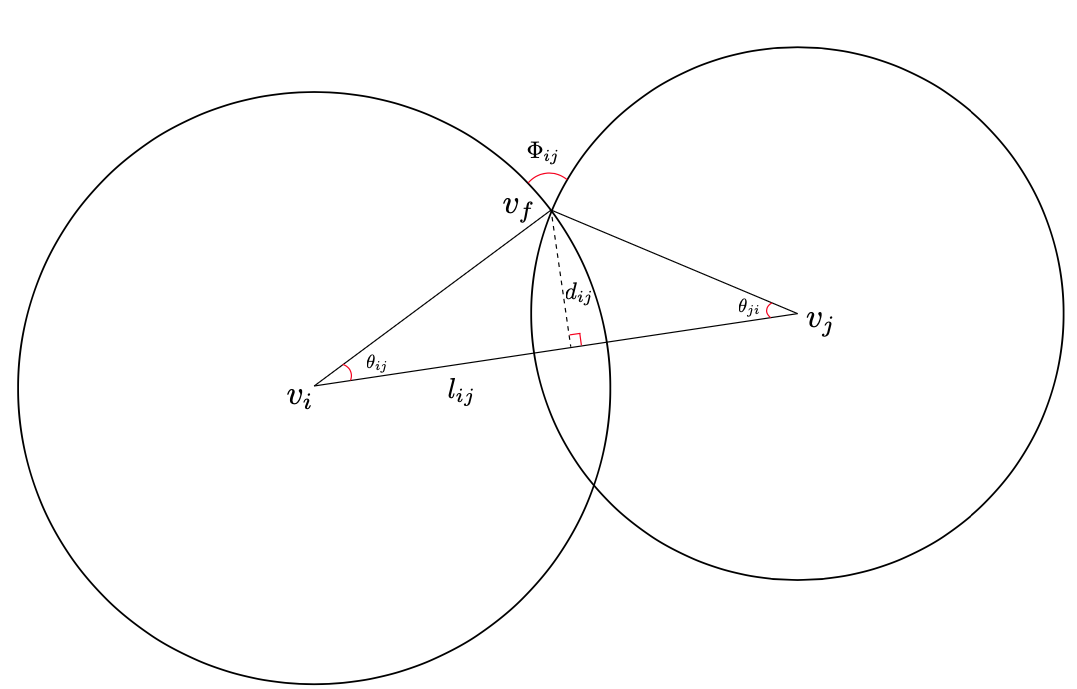}
	\caption{Two-circle configuration}
	\label{fig:twocircleconf}
\end{figure}

\begin{lemma} (\cite[Lemma 2.1]{GHZ21})
	In both Euclidean and hyperbolic geometries, given $\Theta_{ij} \in (0, \pi)$ and two positive numbers $r_i, r_j$, there exists a configuration of two intersecting circles as shown in Figure~\ref{fig:twocircleconf}, which is unique up to isometries, with radii $r_i, r_j$ and meeting at exterior intersection angle $\Theta_{ij}$.
\end{lemma}

Set $u_i = \ln r_i$ in the Euclidean background geometry and $u_i = \ln \tanh \frac{r_i}{2}$ in the hyperbolic background geometry. We have:

\begin{lemma}(\cite[Lemma 2.2]{GHZ21})\label{vari}
	Let $\Theta_{ij} \in (0, \pi)$ be fixed. In hyperbolic background geometry,
	\[
	\frac{\partial \theta_{ij}}{\partial u_i} = -\frac{\cosh l_{ij} \sinh d_{ij}}{\sinh l_{ij}} < 0, \quad \frac{\partial \theta_{ij}}{\partial u_j} = \frac{\partial \theta_{ji}}{\partial u_i} = \frac{\sinh d_{ij}}{\sinh l_{ij}} > 0, \quad \frac{\partial \operatorname{Area}(\triangle v_i v_j v_f)}{\partial u_i} > 0.
	\]
	In Euclidean background geometry,
	\[
	\frac{\partial \theta_{ij}}{\partial u_i} = -\frac{d_{ij}}{l_{ij}} < 0, \quad \frac{\partial \theta_{ij}}{\partial u_j} = \frac{\partial \theta_{ji}}{\partial u_i} = \frac{d_{ij}}{l_{ij}} > 0.
	\]
\end{lemma}

\begin{lemma}\label{key1}
	In both hyperbolic and Euclidean background geometries, let $\Theta_{ij} \in (0, \pi)$ be fixed. Then
	\[
	0 < \frac{\partial \theta_{ij}}{\partial u_j} \leq C \theta_{ij},
	\]
	where $C$ depends only on $\Theta_{ij}$.
\end{lemma}

\begin{proof}
	In hyperbolic background geometry,
	\[
	\frac{\partial \theta_{ij}}{\partial u_j} = \frac{\sinh d_{ij}}{\sinh l_{ij}} = \frac{\sinh r_i \sin  \theta_{ij}}{\sinh l_{ij}}.
	\]
	Since $\sin  \theta_{ij} \leq  \theta_{ij}$ and $\sinh r_i \sin \Theta_{ij} \leq \sinh l_{ij}$, we get
	\[
	\frac{\partial  \theta_{ij}}{\partial u_j} \leq \frac{1}{\sin \Theta_{ij}}  \theta_{ij}.
	\]
	
	In Euclidean geometry,
	\[
	\frac{\partial  \theta_{ij}}{\partial u_j} = \frac{d_{ij}}{l_{ij}} = \frac{r_i \sin  \theta_{ij}}{l_{ij}}.
	\]
	Similarly, using $\sin  \theta_{ij} \leq  \theta_{ij}$ and $r_i \sin \Theta_{ij} \leq l_{ij}$, we get
	\[
	\frac{\partial  \theta_{ij}}{\partial u_j} \leq \frac{1}{\sin \Theta_{ij}}  \theta_{ij}.
	\]
\end{proof}
By same reason, we have 
\begin{lemma}\label{key3}
	In hyperbolic geometry, let $\Theta_{ij} \in (0, \pi)$ be fixed. Then
	\[
	-C  \theta_{ij} \cosh(r_i + r_j) \leq \frac{\partial  \theta_{ij}}{\partial u_i} < 0,
	\]
	where $C$ depends only on $\Theta_{ij}$.
\end{lemma}

\begin{lemma}\label{key2}
	In Euclidean geometry, let $\Theta_{ij} \in (0, \pi)$ be fixed. Then
	\[
	-C  \theta_{ij} \leq \frac{\partial  \theta_{ij}}{\partial u_i} < 0,
	\]
	where $C$ depends only on $\Theta_{ij}$.
\end{lemma}

\begin{lemma}(\cite[Lemma 2.3]{GHZ21}\label{hyperconstr}
	In hyperbolic geometry, for any $\epsilon > 0$, there exists a positive number $L(\epsilon)$ such that $ \theta_{ij} < \epsilon$ whenever $r_i > L(\epsilon)$.
\end{lemma}

This is equivalent to the following reformulation:

\begin{lemma}\label{hyperconstr_u}
	In hyperbolic geometry, for any $\epsilon > 0$, there exists a negative number $\delta(\epsilon)$ such that $ \theta_{ij} < \epsilon$ whenever $0 > u_i > \delta(\epsilon)$.
\end{lemma}

\begin{lemma}\label{eucconstri}
	In Euclidean geometry, for any $C > 1$, there exists an $\epsilon > 0$ such that if $|u_i' - u_i| \leq \epsilon$, $|u_j' - u_j| \leq \epsilon$, then
	\[
	 \theta_{ij}(u_i', u_j') \leq C  \theta_{ij}(u_i, u_j).
	\]
\end{lemma}

\begin{proof}
	Let $f(u_i, u_j) = \ln  \theta_{ij}(u_i, u_j)$. From Lemmas~\ref{key1} and~\ref{key2}, we have
	\begin{equation}\label{aiqa}
			\left| \frac{\partial f}{\partial u_i} \right| = \left| \frac{1}{ \theta_{ij}} \frac{\partial  \theta_{ij}}{\partial u_i} \right| \leq C_0, \quad 
		\left| \frac{\partial f}{\partial u_j} \right| = \left| \frac{1}{ \theta_{ij}} \frac{\partial  \theta_{ij}}{\partial u_j} \right| \leq C_0,
	\end{equation}
	for some constant $C_0$ depending only on $\Theta_{ij}$. Then
	\begin{equation}\label{aiqa2}
			\ln \frac{ \theta_{ij}(u_i', u_j')}{ \theta_{ij}(u_i, u_j)} = f(u_i', u_j') - f(u_i, u_j) = \int_{(u_i, u_j)}^{(u_i', u_j')} \left( \frac{\partial f}{\partial u_i} du_i + \frac{\partial f}{\partial u_j} du_j \right),
	\end{equation}
	so if $|u_i' - u_i| < \epsilon$ and $|u_j' - u_j| < \epsilon$, combining \eqref{aiqa} and \eqref{aiqa2}, then
	\[
	\left| \ln \frac{ \theta_{ij}(u_i', u_j')}{ \theta_{ij}(u_i, u_j)} \right| \leq 2C_0 \epsilon,
	\]
	which implies
	\[
	 \theta_{ij}(u_i', u_j') \leq e^{2C_0 \epsilon}  \theta_{ij}(u_i, u_j).
	\]
\end{proof}
By same reason, we have 
\begin{lemma}\label{hyperconstri}
	In hyperbolic geometry, for any $C > 1$, there exists an $\epsilon > 0$ such that if $|u_j' - u_j| \leq \epsilon$, then
	\[
	 \theta_{ij}(u_i, u_j') \leq C  \theta_{ij}(u_i, u_j).
	\]
\end{lemma}

\subsection{Characters of ideal circle patterns}
Let \(\mathcal{D} = (V, E, F)\) be an infinite cellular decomposition on a non-compact surface \(S\), and let \(\Theta \in (0, \pi)^E\) be a weight function on the edges. We define the characters of ideal circle patterns as follows.

\begin{definition}
	For a weighted cellular decomposed surface \((S, \mathcal{D}, \Theta)\), the \emph{character} \(\mathcal{L}_i(\mathcal{D}, \Theta)\) at each vertex \(v_i \in V\) is defined by
	\[
	\mathcal{L}_i(\mathcal{D}, \Theta) = \sum_{j \sim i} \Theta_{ij},
	\]
	where the sum runs over all vertices \(v_j\) adjacent to \(v_i\).
\end{definition}

\begin{lemma}
	For a weighted cellular decomposed surface \((S, \mathcal{D}, \Theta)\), the character \(\mathcal{L}_i(\mathcal{D}, \Theta)\) coincides with the cone angle \(\sigma(v_i)\) of the circle pattern with radii \(r = (1, 1, 1, \ldots)\) in the Euclidean background geometry.
\end{lemma}
\begin{proof}
	Since \(r_i = 1\) for every vertex \(v_i \in V\), it follows that \( \theta_{ij} = \Theta_{ij} / 2\). By the definition of the character, we have
	\[
	\mathcal{L}_i(\mathcal{D}, \Theta) = \sum_{j \sim i} \Theta_{ij} = 2 \sum_{j \sim i} \theta_{ij} = \sigma(v_i).
	\]
\end{proof}

To facilitate the computation of the character, we use the following lemmas. (See Figure \ref{fig:twocircleconf} for the two-circle configuration.) When \((r_i, r_j) = (t, t)\), we have:

\begin{lemma}\label{character_theta_dec}
	For a weighted cellular decomposed surface \((S, \mathcal{D}, \Theta)\):
	\begin{itemize}
		\item In Euclidean background geometry,
		\[
		\theta_{ij}(t, t) = \frac{\Theta_{ij}}{2}, \quad \forall t > 0.
		\]
		\item In hyperbolic background geometry, \(\theta_{ij}(t, t)\) is a continuously differentiable and strictly decreasing function for \(t \in (0, +\infty)\). Moreover,
		\[
		\lim_{t \to 0} \theta_{ij}(t, t) = \frac{\Theta_{ij}}{2}, \quad \text{and} \quad \lim_{t \to +\infty} \theta_{ij}(t, t) = 0.
		\]
	\end{itemize}
\end{lemma}
\begin{proof}
	We only prove the hyperbolic case, since the Euclidean case is immediate.
	
	In the hyperbolic case, by Lemma \ref{vari}, \(\theta_{ij}(t, t)\) is continuously differentiable and strictly decreasing in \(t \in (0, +\infty)\). Recall that
	\[
	\cosh l_{ij} = \cosh r_i \cosh r_j + \cos \Theta_{ij} \sinh r_i \sinh r_j,
	\]
	and
	\[
	\frac{\sinh l_{ij}}{\sin \Theta_{ij}} = \frac{\sinh r_j}{\sin \theta_{ij}}.
	\]
	From these, we deduce
	\[
	\sin^2 \theta_{ij}(t, t) = \frac{\sin^2 \Theta_{ij}}{2(1 + \cos \Theta_{ij}) + (1 + \cos \Theta_{ij})^2 \sinh^2 t}.
	\]
	Evaluating the limit as \(t \to 0\),
	\[
	\lim_{t \to 0} \sin^2 \theta_{ij}(t, t) = \frac{\sin^2 \Theta_{ij}}{2(1 + \cos \Theta_{ij})} = \sin^2 \frac{\Theta_{ij}}{2},
	\]
	which implies
	\[
	\lim_{t \to 0} \theta_{ij}(t, t) = \frac{\Theta_{ij}}{2}.
	\]
	Also, as \(t \to +\infty\), the denominator tends to infinity, so
	\[
	\lim_{t \to +\infty} \theta_{ij}(t, t) = 0.
	\]
\end{proof}

Recall that \(u_i = \ln r_i\) in Euclidean background geometry and \(u_i = \ln \tanh \frac{r_i}{2}\) in hyperbolic background geometry. From Lemma \ref{vari}, we derive the following comparison principles:

\begin{lemma}\label{character_min_ineq}
	In both Euclidean and hyperbolic background geometries, the following holds:
	\begin{itemize}
		\item[(1)] If \(r_i \leq r_j\), then
		\[
		\theta_{ij}(r_i, r_j) \geq \theta_{ij}(r_i, r_i), \quad \text{and} \quad \theta_{ij}(r_i, r_j) \geq \theta_{ij}(r_j, r_j).
		\]
		\item[(2)] If \(r_i \geq r_j\), then
		\[
		\theta_{ij}(r_i, r_j) \leq \theta_{ij}(r_i, r_i), \quad \text{and} \quad \theta_{ij}(r_i, r_j) \leq \theta_{ij}(r_j, r_j).
		\]
	\end{itemize}
\end{lemma}

\subsection{A maximum principle for infinite graphs}
Given an undirected infinite graph \( G = (V, E) \), where \( V \) and \( E \) are the vertex and edge sets of \( G \) respectively, we write \( v_i \sim v_j \) if \( v_i \) and \( v_j \) are connected by an edge in \( E \). For a function \( f: V \to \mathbb{R} \), the discrete Laplacian operator \(\Delta_G\) is given by
\begin{equation}\label{delta_lap}
	\Delta_G f_i = \sum_{j\sim i} \omega_{ij} (f_j - f_i),
\end{equation}
where \(\omega_{ij}\) is the weight on the edge \(E\).

\begin{lemma}[Maximum principle for infinite graphs]\label{maximumprin}
	Let \( G = (V, E) \) be an undirected infinite graph, and \(\omega_{ij}(t)\) be weights on \(E\) for \( t \in [0, T] \) with \( T > 0 \). Assume there exists a uniform constant \( C \) such that
	\[
	\sum_{j\sim i} \omega_{ij}(t) < C,
	\]
	for any \( v_i \in V \) and any \( t \in [0, T] \). Suppose a function \( f: V \times [0, T] \to \mathbb{R} \) satisfies
	\[
	\frac{d f}{d t} \leq \Delta_{\omega(t)} f + g f,
	\]
	where \( g \leq C_0 \) for some constant \( C_0 \). If \( f \) is bounded on \( V \times [0, T] \) and \( f(0) \leq 0 \), then
	\[
	f(t) \leq 0, \quad \forall t \in [0, T].
	\]
\end{lemma}

\begin{proof}
	Without loss of generality, we may assume \( C_0 = 0 \). Indeed, setting \(\tilde{f} = e^{-C_0 t} f\) and \(\tilde{g} = g - C_0\), the problem reduces to this case.
	
	Choose a vertex \( v_0 \in V \), and define the distance function
	\[
	d_i^{[v_0]} := d(v_i, v_0),
	\]
	where \( d(v_i, v_0) \) is the graph distance in \( G \). Then,
	\[
	\Delta_{\omega(t)} d_i^{[v_0]} \leq \sum_{j\sim i} \omega_{ij}(t) < C,
	\]
	for all \( v_i \in V \) and \( t \in [0, T] \).
	
	For \(\delta > 0\), define
	\[
	f^\delta := f - \delta d^{[v_0]} - C \delta t.
	\]
	We compute
	\[
	\begin{aligned}
		\frac{d f^\delta}{d t} &= \frac{d f}{d t} - C \delta \\
		&\leq \Delta_{\omega(t)} f + g f - C \delta \\
		&= \Delta_{\omega(t)} f^\delta + \delta \Delta_{\omega(t)} d^{[v_0]} + g f - C \delta \\
		&< \Delta_{\omega(t)} f^\delta + g f \\
		&\leq \Delta_{\omega(t)} f^\delta + g f^\delta.
	\end{aligned}
	\]
	
	Since \( f \) is bounded, we have
	\begin{equation}\label{confd}
		f_i^\delta \to -\infty \quad \text{as} \quad d_i^{[v_0]} \to \infty.
	\end{equation}
	
	We claim that \( f^\delta \leq 0 \) on \( V \times [0, T] \). Suppose not. Then by \eqref{confd} there exists \( (v_i, t_0) \in V \times [0, T] \) where \( f^\delta \) attains a positive maximum, i.e., \( f_i^\delta(t_0) > 0 \).
	
	Consider two cases:
	\begin{itemize}
		\item If \( t_0 = 0 \), then
		\[
		0 < f_i^\delta(t_0) \leq f_i(0) \leq 0,
		\]
		which is a contradiction.
		\item If \( t_0 > 0 \), then
		\[
		\frac{d f_i^\delta}{d t}(t_0) \geq 0, \quad \text{and} \quad \Delta_{\omega(t_0)} f_i^\delta (t_0) \leq 0.
		\]
		Hence,
		\[
		0 \leq \frac{d f_i^\delta}{d t}(t_0) < \Delta_{\omega(t_0)} f_i^\delta(t_0) + g f_i^\delta(t_0) \leq 0,
		\]
		which is a contradiction.
	\end{itemize}
	
	Therefore, \( f^\delta \leq 0 \) for all \( (v_i, t) \). Since $\delta$ is arbitrary, it follows  \( f \leq 0 \).
\end{proof}

\begin{corollary}\label{cor0}
	If \( f(t) \) is a bounded solution to
	\[
	\frac{d f}{d t} = \Delta_{\omega(t)} f + g f,
	\]
	with initial condition \( f(0) = 0 \), and \( g \leq C_0 \) for some constant \( C_0 \), then
	\[
	f(t) = 0, \quad \forall t \in [0, T].
	\]
\end{corollary}

\begin{corollary}
	If \( f(t) \) is a bounded solution to
	\begin{equation}\label{equequi}
		\frac{d f}{d t} = \Delta_{\omega(t)} f + g f,
	\end{equation}
	with
	\[
	C_1 \leq f_i(0) \leq C_2, \quad \forall i \in V,
	\]
	where \( C_1<0 ,  C_2 > 0 \) and \( g \leq 0 \), then
	\[
	C_1 \leq f_i(t) \leq C_2, \quad \forall i \in V, \quad t \in [0, T].
	\]
\end{corollary}

\begin{proof}
	Set \( h := f - C_2 \). Then \( h(0) \leq 0 \) and from \eqref{equequi} we have
	\[
	\frac{d h}{d t} = \Delta_{\omega(t)} h + g h + g C_2 \leq \Delta_{\omega(t)} h + g h.
	\]
	By Lemma \ref{maximumprin},
	\[
	h(t) \leq 0, \quad \forall t \in [0, T],
	\]
	which implies
	\[
	f(t) \leq C_2, \quad \forall t \in [0, T].
	\]
	
	Similarly, setting \( k :=C_1-f \) and applying the same argument yields
	\[
	C_1 \leq f(t), \quad \forall t \in [0, T].
	\]
\end{proof}

\subsection{Discrete Laplacian and analysis on lattice}

Consider the infinite cellular decomposition \( \mathcal{D} = (V, E, F) \) of the plane \(\mathbb{R}^2\), where the 1-dimensional skeleton \(G = (V, E)\) is given by
\[
V = \{ v_{m,n} = (m,n) \in \mathbb{R}^2 \mid m, n \in \mathbb{Z} \},
\]
and
\[
E = \{ [v_{m,n}, v_{m', n'}] \mid |v_{m,n} - v_{m', n'}| = 1 \}.
\]

For \( p \in [1, \infty] \), let \( l^p(V) \) denote the space of functions \( u : V \to \mathbb{R} \) endowed with the counting measure, equipped with the norm
\[
N_{0,p}(u):=\|u\|_{l^p}=
\begin{cases}
	\Big( \sum\limits_{(m,n) \in \mathbb{Z}^2} |u_{m,n}|^p \Big)^{\frac{1}{p}}, & 1 \leq p < \infty, \\[8pt]
	\hspace{2.3em}\sup\limits_{(m,n) \in \mathbb{Z}^2} |u_{m,n}|, & p = \infty.
\end{cases}
\]
Note that the \( l^p \)-norms satisfy the monotonicity property: for \(1 \leq p \leq q \leq \infty\),
\[
\|u\|_{l^q} \leq \|u\|_{l^p}.
\]

For convenience, we write \( u = (u_{m,n})_{m,n \in \mathbb{Z}} \). Define the following discrete difference operators acting on functions \(u\):
\[
\begin{cases}
	D_1 u_{m,n} = u_{m+1,n} - u_{m,n}, \\[6pt]
	D_3 u_{m,n} = -D_1 u_{m-1,n} = u_{m-1,n} - u_{m,n}, \\[6pt]
	D_2 u_{m,n} = u_{m,n+1} - u_{m,n}, \\[6pt]
	D_4 u_{m,n} = -D_2 u_{m,n-1} = u_{m,n-1} - u_{m,n}.
\end{cases}
\]

We refer to these four operators collectively as the \emph{basic difference operators} on \(G\), and denote the set
\[
\mathcal{D} = \{ D_1, D_2, D_3, D_4 \}.
\]

It is straightforward to verify that each \(D_i\) maps \(l^p(V)\) into itself and that the operators commute pairwise, i.e.,
\begin{equation}\label{commute}
	D_i D_j u = D_j D_i u, \quad \forall u \in l^p(V), \quad \forall D_i, D_j \in \mathcal{D}.
\end{equation}

For \( p \in [1, \infty) \), define the seminorms
\begin{equation}\label{N1p}
	N_{1,p}(u) := \bigg( \sum_{D_i \in \mathcal{D}} \|D_i u\|_{l^p}^p \bigg)^{\frac{1}{p}}
\end{equation}
and
\begin{equation}\label{N2p}
	N_{2,p}(u) := \bigg( \sum_{D_i, D_j \in \mathcal{D}} \|D_i D_j u\|_{l^p}^p \bigg)^{\frac{1}{p}}.
\end{equation}

One readily checks that
\begin{equation}\label{N1N0}
	N_{1,p}^p(u) = \sum_{D_i \in \mathcal{D}} N_{0,p}^p(D_i u).
\end{equation}
and 
\begin{equation}\label{N2N1}
	N_{2,p}^p(u) = \sum_{D_i \in \mathcal{D}} N_{1,p}^p(D_i u).
\end{equation}

It follows from the boundedness of finite differences that there exists a universal constant \(C_1 > 0\) such that
\begin{equation}\label{sob}
	N_{1,p}^p(u) \leq C_1 \|u\|_{l^p}^p, \quad
	N_{2,p}^p(u) \leq C_1 \|u\|_{l^p}^p, \quad \forall u \in l^p(V).
\end{equation}

The standard inner product on \(l^2(V)\) is given by
\[
(f,g) = \sum_{(m,n) \in \mathbb{Z}^2} f_{m,n} g_{m,n}.
\]

\begin{lemma}
	There is a formula:
	\begin{equation}\label{dfdg}
		\sum_{D_i \in \mathcal{D}}    (D_i f, D_i g) = 2 \sum_{[i,j] \in E} (f(i) - f(j))(g(i) - g(j))
	\end{equation}
\end{lemma}
\begin{proof}
	By definition,
	$$
	\begin{aligned}
			 \sum_{D_i \in \mathcal{D}}    (D_i f, D_i g)&= \sum_{D_i \in \mathcal{D}} \sum_{(m,n)} (D_i f)_{m,n}(D_i g)_{m,n}\\
			 & = \sum_{(m,n)} \sum_{D_i \in \mathcal{D}}  (D_i f)_{m,n}(D_i g)_{m,n}\\
	\end{aligned}
	$$
	By the structure of the discrete derivative operators \(D_i\), each undirected edge \([i,j]\in E\) contributes twice to the total sum, which yields the desired formula~\eqref{dfdg}.
\end{proof}

\begin{corollary}
	For \(p=2\), the seminorm \(N_{1,2}\) has the explicit form
	\begin{equation}\label{n1pformula}
		N_{1,2}^2(u) = 2 \sum_{e = [i,j] \in E} (u(i) - u(j))^2,
	\end{equation}
	where each edge is counted only once.
\end{corollary}
\begin{proof}
	By definition and \ref{dfdg},
	\[
	N_{1,2}^2(u) = \sum_{D_i \in \mathcal{D}} \|D_i u\|_{l^2}^2 = \sum_{D_i \in \mathcal{D}}    (D_i u, D_i u) = 2 \sum_{e = [i,j] \in E} (u(i) - u(j))^2.
	\]
\end{proof}

On the lattice \(G\), consider the discrete Laplacian with constant weights \(w_{ij} = 1\) for all edges \([i,j] \in E\). By definition,
\begin{equation}\label{delta_def}
	\Delta = \sum_{i=1}^4 D_i.
\end{equation}
\begin{lemma}
	There is the discrete Green's formula:
	\begin{equation}\label{stokes1}
		(f, \Delta g) = - \sum_{e = [i,j] \in E} (f(i) - f(j))(g(i) - g(j)), \quad \forall f,g \in \ell^2(V).
	\end{equation}
\end{lemma}

\begin{proof}
	We illustrate the argument using two directional differences, \(D_1\) and \(D_3\), corresponding to horizontal edges in the positive and negative directions, respectively. Note that:
	\[
	\begin{aligned}
		(f, D_1 g) &= \sum_{(m,n) \in \mathbb{Z}^2} f_{m,n} (g_{m+1,n} - g_{m,n}) \\
		&= \sum_{(m,n) \in \mathbb{Z}^2} f_{m-1,n} (g_{m,n} - g_{m-1,n}),
	\end{aligned}
	\]
	after a change of variable \(m \mapsto m-1\) in the summation.
	
	Similarly,
	\[
	(f, D_3 g) = \sum_{(m,n) \in \mathbb{Z}^2} f_{m,n} (g_{m-1,n} - g_{m,n}) = - \sum_{(m,n) \in \mathbb{Z}^2} f_{m,n} (g_{m,n} - g_{m-1,n}).
	\]
	
	Adding the two gives:
	\[
	\begin{aligned}
		(f, D_1 g) + (f, D_3 g)
		&= \sum_{(m,n)} f_{m-1,n}(g_{m,n} - g_{m-1,n}) - \sum_{(m,n)} f_{m,n}(g_{m,n} - g_{m-1,n}) \\
		&= - \sum_{(m,n)} (f_{m,n} - f_{m-1,n})(g_{m,n} - g_{m-1,n}).
	\end{aligned}
	\]
	
	Analogous identities hold for the vertical directions \(D_2\) and \(D_4\). Summing over all four directions, and observing that each undirected edge \([i,j]\in E\) is counted exactly once, we obtain:
	\[
	(f, \Delta g) = - \sum_{[i,j] \in E} (f(i) - f(j))(g(i) - g(j)),
	\]
	which proves the discrete Green's formula \eqref{stokes1}.
\end{proof}

\begin{corollary}
	For \(u \in l^2(V)\),
	\begin{equation}\label{fdeltag}
			(f, \Delta g) = -\frac{1}{2}\sum_{D_i \in \mathcal{D}}    (D_i f, D_i g) 
	\end{equation}
\end{corollary}
\begin{proof}
	Combining \eqref{stokes1} and \eqref{dfdg}, we can get \eqref{fdeltag}.
\end{proof}

\begin{corollary}
	In particular, for \(u \in l^2(V)\),
	\begin{equation}\label{stokes}
		(u, \Delta u) = -\frac{1}{2} N_{1,2}^2(u).
	\end{equation}
\end{corollary}
\begin{proof}
	Let $f=g=u$, combining \eqref{stokes1} and \eqref{n1pformula}, we can get \eqref{stokes}.
\end{proof}

\begin{lemma}
	For all \( u \in l^2(V) \),
	\begin{equation}\label{N02N22}
		N_{0,2}(\Delta u)=  \frac{1}{2} N_{2,2}(u)
	\end{equation}
\end{lemma}
\begin{proof}
	By the definition,
	\begin{equation}\label{hdqdqdq}
		N_{0,2}^2(\Delta u)= \|\Delta  u\|_{l^2}^2= (\Delta u, \Delta  u)
	\end{equation}
	From \eqref{fdeltag} and \eqref{stokes}, 
	\begin{equation}
			\begin{aligned}\label{fenkfkq}
			(\Delta u, \Delta  u)& = -\frac{1}{2}\sum_{D_i \in \mathcal{D}}    ( D_i \Delta  u , D_i u) \\
			&=  -\frac{1}{2}\sum_{D_i \in \mathcal{D}}    (  D_i u,  \Delta  (D_i u))\\
			&= \frac{1}{4} \sum_{D_i \in \mathcal{D}}   N_{1,2}^2(D_i  u).
		\end{aligned}
	\end{equation}
	From \eqref{N2N1}, and combining \eqref{hdqdqdq} and \eqref{fenkfkq}, we have 
	$$
	N_{0,2}^2(\Delta u) = (\Delta u, \Delta  u)= \frac{1}{4} \sum_{D_i \in \mathcal{D}}  = \frac{1}{4} N_{1,2}^2(D_i  u)=  \frac{1}{4} N_{2,2}^2(u).
	$$
\end{proof}

\section{Infinite combinatorial Ricci flows for ideal circle patterns}\label{Sec:3}

\subsection{The long time existence of the flow}

Recall that \(\mathcal{D} = (V, E, F)\) is an infinite cellular decomposition of a surface \(S\). Let \(r \in \mathbb{R}_{+}^V\) be a radius vector for \(\mathcal{D}\). We define
\begin{equation}\label{tr1}
	u_i = \ln \tanh \left(\frac{r_i}{2}\right), \quad \forall v_i \in V
\end{equation}
in hyperbolic background geometry, and
\begin{equation}\label{tr2}
	u_i = \ln r_i, \quad \forall v_i \in V
\end{equation}
in Euclidean background geometry.

The combinatorial Ricci flow \eqref{flow} is equivalent to the system
\begin{equation}\label{flow2}
	\frac{\mathrm{d} u_i}{\mathrm{d} t} = -K_i, \quad \forall v_i \in V.
\end{equation}
Since \(V\) is infinite, the classical Picard–Lindelöf theorem does not directly apply to guarantee existence and uniqueness of solutions to \eqref{flow2}.

To address this, we apply the Arzelà–Ascoli theorem. Consider an exhaustive sequence of finite, simple, connected cellular decompositions
\[
\{\mathcal{D}^{[n]} = (V^{[n]}, E^{[n]}, F^{[n]})\}_{n=1}^{\infty}
\]
satisfying
\[
\mathcal{D}^{[n]} \subset \mathcal{D}^{[n+1]}, \quad \text{and} \quad \bigcup_{n=1}^\infty \mathcal{D}^{[n]} = \mathcal{D}.
\]
The sequence must exist, since we have assumed $d(v_i, v_j) < +\infty$ for all $v_i, v_j \in V$.

We study the corresponding finite-dimensional flows on \(\mathcal{D}\):
\begin{equation}\label{flow_finite}
	\begin{cases}
		\dfrac{\mathrm{d} u_i^{[n]}(t)}{\mathrm{d} t} = -K_i^{[n]}, & \forall v_i \in V^{[n]}, \quad t > 0, \\[7pt]
		u_i^{[n]}(0) = u_i(0), & \forall v_i \in V^{[n]}, \\[7pt]
		u_i^{[n]}(t) = u_i(0), & \forall v_i \notin V^{[n]}, \quad t \geq 0,
	\end{cases}
\end{equation}
where \(K_i^{[n]}(t) = K_i(u^{[n]}(t))\). By definition, for any \(v_i \in V\), the curvature \(K_i^{[n]}\) depends only on \(u_j^{[n]}\) with \(v_j = v_i\) or \(v_j \sim v_i\). Since \(V^{[n]}\) is finite, \eqref{flow_finite} is a finite-dimensional ODE system. By the Picard–Lindelöf theorem, it admits a unique solution on some interval.

\begin{lemma}\label{flow_finite_exi}
	For each \(n\), the flow \eqref{flow_finite} admits a unique solution \(u^{[n]}(t)\) that exists for all \(t \geq 0\).
\end{lemma}
\begin{proof}
	Since each \(K_i^{[n]}\) is a smooth function of \(u^{[n]}\), it is locally Lipschitz continuous. By classical ODE theory, there exists a unique solution \(u^{[n]}(t)\) on some interval \([0, \epsilon)\), \(\epsilon > 0\). To extend the solution to \([0, \infty)\), it suffices to prove uniform boundedness of \(u^{[n]}(t)\) on finite time intervals.
	
	\emph{Part I: we prove that \( u^{[n]}(t) \) has a lower bound in $V^{[n]}$.}
	
	By definition, the curvature satisfies
	\[
	2\pi (1 - \deg(v_i)) \leq K_i^{[n]} < 2\pi,
	\]
	where \(\deg(v_i)\) is the degree of vertex \(v_i\) in \(\mathcal{D}^{[n]}\). Since \(\mathcal{D}\) is finite, there exists a uniform constant \(c > 0\) such that \(|K_i^{[n]}| < c\) for all \(i\in V^{[n]} \).
	
	From \eqref{flow_finite}, integrating over time yields
	\begin{equation}\label{keyine}
		-c t \leq u_i^{[n]}(t) - u_i(0) \leq c t, \quad \forall t \in [0,T],\quad \forall i\in V^{[n]},
	\end{equation}
	which implies
	\[
	u_i^{[n]}(t) \geq u_i(0) - c T, \quad \forall i\in V^{[n]}.
	\]
	
	\emph{Part II: we prove that \( u^{[n]}(t) \) has an up bound in $V^{[n]}$.}
	
	In the hyperbolic setting (using \eqref{tr1}), we need to show there exists a constant \(c_1 < 0\) such that \(u_j^{[n]}(t) < c_1\) for all \(j \in V^{[n]}\). By Lemma \ref{hyperconstr_u}, there exists \(\delta < 0\) such that if \(\delta < u_j^{[n]}(t) < 0\) for all \(v_j \in V^{[n]}\), then
	\[
	K_i^{[n]} \geq \frac{\pi}{2}, \quad \forall i \in V^{[n]}.
	\]
	From \eqref{flow_finite}, this implies that \(u_j^{[n]}(t)\) is decreasing when \(u_j^{[n]}(t) \in (\delta, 0)\), ensuring the existence of such a uniform upper bound \(c_1 < 0\).
	
	In the Euclidean case (using \eqref{tr2}), the upper bound follows directly from \eqref{keyine}:
	\[
	u_i^{[n]}(t) \leq u_i(0) + c T, \quad \forall i\in V^{[n]}.
	\]
\end{proof}

Since \( u^{[n]}(t) \) evolves according to the flow \eqref{flow_finite}, for any \( v_i \in V^{[n]} \), we have 
\begin{equation}\label{curvature_finite_flow}
	\frac{d}{dt} K_i^{[n]}(t) = - \frac{\partial K_i^{[n]}}{\partial u_i} K_i^{[n]}(t) - \sum_{\substack{v_j \sim v_i \\ v_j \in V^{[n]}}} \frac{\partial K_i^{[n]}}{\partial u_j} K_j^{[n]}(t).
\end{equation}

Let \( V(v_i) \) denote the set of vertices \( v_k \) such that either \( v_k = v_i \) or \( v_k \sim v_i \), and let \( E(v_i) \) denote the set of edges incident to \( v_i \).

\begin{theorem}\label{existence_u}
	The flow \eqref{flow2} exists for all time \( t \geq 0 \).
\end{theorem}

\begin{proof}
	Fix any \( T > 0 \) and consider the interval \( t \in [0, T] \).
	
	For any vertex \( v_i \in V \), there exists a sufficiently large \( N \) such that \( v_i \in V^{[n]} \) for all \( n \geq N \). Denote \( K_i^{[n]}(t) := K_i(u^{[n]}(t)) \). By the definition of \( K_i \), we have the uniform bound
	\begin{equation}\label{e1}
		|K_i^{[n]}(t)| \leq 2\pi (1 + \deg(v_i)), \quad t \in [0, T].
	\end{equation}
	
	From the flow equation \eqref{flow_finite}, it follows that
	\begin{equation}\label{e2}
		u_i^{[n]}(t) \geq u_i(0) - 2\pi (1 + \deg(v_i)) T, \quad t \in [0, T].
	\end{equation}
	
	In the hyperbolic background geometry, by Lemma \ref{hyperconstr_u} and the argument in Part II of the proof of Lemma \ref{flow_finite_exi}, for all \( v_j \in V(v_i) \) we have the uniform upper bound
	\begin{equation}\label{hyper_finite_constr}
		u_j^{[n]}(t) < c_1 < 0,
	\end{equation}
	where the constant \( c_1 \) depends only on the degrees of vertices in \( V(v_i) \).
	
	In the Euclidean background geometry, similarly,
	\begin{equation}\label{e9}
		u_i^{[n]}(t) \leq u_i(0) + 2\pi (1 + \deg(v_i)) T, \quad t \in [0, T].
	\end{equation}
	
	Combining \eqref{e1}, \eqref{e2}, \eqref{hyper_finite_constr}, and \eqref{e9}, we deduce the uniform estimate
	\begin{equation}\label{keyesti}
		\begin{aligned}
			\|u_i^{[n]}(t)\|_{C^1[0,T]} 
			&= \sup_{t \in [0,T]} |u_i^{[n]}(t)| + \sup_{t \in [0,T]} \Big| \frac{d u_i^{[n]}(t)}{dt} \Big| \\[5pt]
			&= \sup_{t \in [0,T]} |u_i^{[n]}(t)| + \sup_{t \in [0,T]} |K_i^{[n]}(t)| \\[5pt]
			&\leq |u_i(0)| + 2\pi (1 + \deg(v_i)) (T + 1).
		\end{aligned}
	\end{equation}
	
	Since \( K_i \) is smooth in \( u \), the partial derivatives \( \frac{\partial K_i}{\partial u_i} \) and \( \frac{\partial K_i}{\partial u_j} \) are also smooth. Using the bounds in \eqref{e2}, \eqref{hyper_finite_constr}, and \eqref{e9}, there exists a constant \( C > 0 \), depending only on \( u_i(0) \), \( \deg(v_i) \), and (in the hyperbolic case) the degrees of vertices in \( V(v_i) \), such that for all \( j \sim i \),
	\[
	\max \left\{ \bigg| \frac{\partial K_i^{[n]}}{\partial u_i^{[n]}} \bigg|,\; \bigg| \frac{\partial K_i^{[n]}}{\partial u_j^{[n]}} \bigg| \right\} \leq C.
	\]
	
	Therefore, from \eqref{curvature_finite_flow} and \eqref{e1}, there exists a constant \( C' > 0 \) such that
	\begin{equation}\label{keyesti2}
		\bigg| \frac{d}{dt} K_i^{[n]}(t) \bigg| \leq C'.
	\end{equation}
	
	By \eqref{e1}, \eqref{keyesti}, and \eqref{keyesti2}, for each vertex \( v_i \in V \), the sequence \( \{ u_i^{[n]}(t) \}_{n=1}^\infty \) is uniformly bounded and equicontinuous on \([0,T]\) for all sufficiently large \( n \geq N \) (depending on \( v_i \)).
	
	Since \( V \) is countable, applying the Arzel\`a–Ascoli theorem and a diagonal argument, there exists a subsequence \( \{ u^{[n_{m(T)}]}(t) \}_{m(T)=1}^\infty \) that converges uniformly to some function \( \hat{u}_T(t) \) on \([0,T]\), with derivatives also converging uniformly. Hence, \( \hat{u}_T(t) \) satisfies the flow \eqref{flow2} on \([0,T]\).
	
	Let \( \{ T_\tau \}_{\tau=1}^\infty \) be an increasing sequence with \( T_\tau \to +\infty \). Repeating the diagonal argument on these intervals, we obtain a subsequence \( \{ u^{[n_m]}(t) \}_{m=1}^\infty \) converging to \( \hat{u}(t) \) on every finite interval \( [0,T_\tau] \), and thus on \([0,+\infty)\). Therefore, \( \hat{u}(t) \) exists globally and solves the flow \eqref{flow2} for all \( t \geq 0 \).
	
	Finally, in the hyperbolic background geometry, the configuration space is \( (-\infty,0)^V \) by \eqref{tr1}. Since \eqref{hyper_finite_constr} implies \( u_i^{[n]}(t) \leq c_1 < 0 \) uniformly, the limit satisfies \( \hat{u}_i(t) < 0 \) for all \( v_i \in V \) and all \( t \geq 0 \), ensuring \( \hat{u}(t) \) remains in the configuration space.
	
\end{proof}

\subsection{The uniqueness of the flow}
\begin{theorem}\label{uniq}
	Suppose $u^{[1]}(t)$ and $u^{[2]}(t)$ are two solutions of equation \eqref{flow2}, and that $K(u^{[1]}(t))$ and $K(u^{[2]}(t))$ are uniformly bounded. Then
	\[
	u^{[1]}(t) \equiv u^{[2]}(t) \quad \text{for all } t \in [0,+\infty).
	\]
\end{theorem}

\begin{proof}
	Define $f(t) = u^{[1]}(t) - u^{[2]}(t)$ and for $s \in [0,1]$ set 
	\[
	u^{[s]}(t) = s u^{[1]}(t) + (1 - s) u^{[2]}(t).
	\]
	Then, from \eqref{flow2},
	\begin{equation}\label{diff_eq}
		\frac{d f(t)}{dt} = - \bigl( K(u^{[1]}(t)) - K(u^{[2]}(t)) \bigr).
	\end{equation}
	By the fundamental theorem of calculus and the chain rule,
	\[
	\begin{aligned}
		K_i(u^{[1]}) - K_i(u^{[2]}) 
		&= \int_{u^{[2]}}^{u^{[1]}} \nabla K_i \cdot d r \\
		&= \int_0^1 \nabla K_i(u^{[s]}) \cdot f \, ds \\
		&= \int_0^1 \bigg( \frac{\partial K_i}{\partial u_i}(u^{[s]}) f_i + \sum_{v_j \sim v_i} \frac{\partial K_i}{\partial u_j}(u^{[s]}) f_j \bigg) ds \\
		&= \sum_{v_j \sim v_i} \bigg( \int_0^1 \frac{\partial K_i}{\partial u_j}(u^{[s]}) ds \bigg) (f_j - f_i) 
		+ \int_0^1 \bigg( \frac{\partial K_i}{\partial u_i} + \sum_{v_j \sim v_i} \frac{\partial K_i}{\partial u_j} \bigg)(u^{[s]}) f_i \, ds.
	\end{aligned}
	\]
	
	Set
	\[
	\omega_{ij}(t) := - \int_0^1 \frac{\partial K_i}{\partial u_j}(u^{[s]}) ds, \quad 
	g_i(t) := - \int_0^1 \bigg( \frac{\partial K_i}{\partial u_i} + \sum_{v_j \sim v_i} \frac{\partial K_i}{\partial u_j} \bigg)(u^{[s]}) ds.
	\]
	
	By Lemma \ref{vari}, we have $\omega_{ij}(t) \geq 0$ and $g_i(t) \leq 0$. Thus \eqref{diff_eq} can be rewritten as
	\[
	\frac{d f}{dt} = \Delta_{\omega(t)} f + g(t) f,
	\]
	where $\Delta_{\omega(t)}$ is the weighted graph Laplacian associated with weights $\omega_{ij}(t)$.
	
	Since $K(u^{[1]}(t))$ and $K(u^{[2]}(t))$ are uniformly bounded, there exist constants $M_1, M_2$ such that
	\[
	|K_i(u^{[1]}(t))| \leq M_1, \quad |K_i(u^{[2]}(t))| \leq M_2, \quad \forall v_i \in V, \quad t \geq 0.
	\]
	Set $M = \max\{M_1, M_2\}$. Then for $t \in \left[0, \frac{\epsilon}{2M}\right]$,
	\[
	|u_i^{[1]}(t) - u_i(0)| \leq \frac{\epsilon}{2}, \quad |u_i^{[2]}(t) - u_i(0)| \leq \frac{\epsilon}{2}, \quad \forall v_i \in V.
	\]
	
	Since $u^{[s]}(t) = s u^{[1]}(t) + (1-s) u^{[2]}(t)$, we have
	\[
	\min\{u_i^{[1]}(t), u_i^{[2]}(t)\} \leq u_i^{[s]}(t) \leq \min\{u_i^{[1]}(t), u_i^{[2]}(t)\} + \epsilon.
	\]
		
From Lemma~\ref{key1}, we have
\begin{equation}\label{hadhqi}
	\begin{aligned}
		\sum_{j\sim i} \omega_{ij}(t)
		&= \sum_{j\sim i} \int_0^1 -\frac{\partial K_i}{\partial u_j}(u^{[s]}) \, ds \\
		&= \sum_{j\sim i} \int_0^1 \bigg| \frac{\partial K_i}{\partial u_j}(u^{[s]}) \bigg| ds \\
		&= 2 \sum_{j\sim i} \int_0^1 \bigg| \frac{\partial \theta_{ij}}{\partial u_j}(u_i^{[s]}, u_j^{[s]}) \bigg| ds \\
		&\leq 2C \sum_{j\sim i} \int_0^1 \theta_{ij}(u_i^{[s]}, u_j^{[s]}) ds,
	\end{aligned}
\end{equation}
where the last inequality follows from Lemma~\ref{vari}. 

For $i \in V$, without loss of generality, we can assume $u_i^{[1]} \leq u_i^{[2]}$. Since \( \partial \theta_{ij} / \partial u_i < 0 \), the function \( \theta_{ij} \) is decreasing in \( u_i \). Together with Lemma \ref{eucconstri} and Lemma \ref{hyperconstri}, it follows
\begin{equation}\label{dqwdhqof}
	\theta_{ij}(u_i^{[s]}, u_j^{[s]}) \leq \theta_{ij}(u_i^{[1]}, u_j^{[s]}) \leq e^{2C\epsilon} \theta_{ij}(u_i^{[1]}, u_j^{[1]}).
\end{equation}
By definition of $K_i$, 
\begin{equation}\label{dnqwdnq}
	\sum_{j\sim i}2 \theta_{ij}(u_i^{[1]}, u_j^{[1]}) \leq |K_i(u^{[1]}(t))|+ 2\pi \leq M+ 2\pi.
\end{equation}

Combining \eqref{hadhqi}, \eqref{dqwdhqof} and \eqref{dnqwdnq}, 
$$
\begin{aligned}
	\sum_{j \sim i} \omega_{i j}(t)& \leq 2C \sum_{j\sim i} \int_0^1 \theta_{ij}(u_i^{[s]}, u_j^{[s]}) ds\\
	&\leq Ce^{2C\epsilon} \sum_{j\sim i} \int_0^1 2\theta_{ij}(u_i^{[1]}, u_j^{[1]}) ds\\
	&\leq Ce^{2C\epsilon} (M+2\pi)\\
\end{aligned}
$$

	Applying Corollary \ref{cor0} (maximum principle for this type of equation), we conclude that
	\[
	f(t) \equiv 0, \quad \text{for } t \in \left[0, \frac{\epsilon}{2M}\right].
	\]
	
	Since $\epsilon > 0$ is arbitrary, the uniqueness extends to all $t \in [0, +\infty)$.

\end{proof}

\begin{corollary}
	If there exists a constant $c$ such that $\deg(v_i) \leq c$ for all $v_i \in V$, then the solution to equation \eqref{flow2} is unique.
\end{corollary}

\begin{proof}
	Since $\deg(v_i) \leq c$, the curvature satisfies
	\[
	2\pi - c \pi \leq K_i(u(t)) < 2\pi, \quad \forall v_i \in V, \quad \forall t \geq 0.
	\]
	Thus $K(u(t))$ is uniformly bounded. The uniqueness then follows directly from Theorem \ref{uniq}.
\end{proof}

\subsection{Convergence of the flow in Hyperbolic background geometry}
\begin{proof}[Proof of Theorem \ref{convergence1}]
	Let $u(t)$ be the solution constructed as the limit of a sequence $\{u^{[n]}(t)\}_{n=1}^\infty$, as in the proof of Theorem \ref{existence_u}, with initial value $u(0)$ satisfying $K(u(0)) \leq 0$.
	
	We first prove that $u(t)$ is non-decreasing. Recall that $K^{[n]}(t) = K(u^{[n]}(t))$, and define
	\[
	f_i^{[n]}(t) = 
	\begin{cases}
		K_i^{[n]}(t) & \text{if } i \in V^{[n]}, \\[5pt]
		0 & \text{if } i \notin V^{[n]}.
	\end{cases}
	\]
	Since $u^{[n]}(t)$ evolves according to the flow \eqref{flow_finite}, for $v_i \in V^{[n]}$ we have
We compute the time derivative of $f_i^{[n]}(t)$ as follows:
\begin{equation} \label{pf1}
	\begin{aligned}
		\frac{d}{dt} f_i^{[n]}(t) 
		&= -\frac{\partial f_i^{[n]}}{\partial u_i} f_i^{[n]}(t) 
		- \sum_{\substack{v_j \sim v_i\\ v_j \in V^{[n]}}} \frac{\partial f_i^{[n]}}{\partial u_j} f_j^{[n]}(t) \\
		&= - \sum_{\substack{v_j \sim v_i\\ v_j \in V^{[n]}}} \frac{\partial f_i^{[n]}}{\partial u_j} \left( f_j^{[n]}(t) - f_i^{[n]}(t) \right) 
		- \bigg( \frac{\partial f_i^{[n]}}{\partial u_i} 
		+ \sum_{\substack{v_j \sim v_i\\ v_j \in V^{[n]}}} \frac{\partial f_i^{[n]}}{\partial u_j} \bigg) f_i^{[n]}(t).
	\end{aligned}
\end{equation}
	For $v_i \notin V^{[n]}$, clearly
	\begin{equation} \label{pf2}
		\frac{d}{dt} f_i^{[n]}(t) = 0.
	\end{equation}
	
We define 
\[
\omega_{ij}(t) =
\begin{cases}
	- \dfrac{\partial f_i^{[n]}}{\partial u_j} & \text{if } i, j \in V^{[n]}, \\
	0 & \text{otherwise,}
\end{cases}
\]
and
\[
g_i(t) =
\begin{cases}
	- \bigg( \dfrac{\partial f_i^{[n]}}{\partial u_i} + \displaystyle\sum_{\substack{v_j \sim v_i\\ v_j \in V^{[n]}}} \dfrac{\partial f_i^{[n]}}{\partial u_j} \bigg) & \text{if } i \in V^{[n]}, \\
	0 & \text{otherwise.}
\end{cases}
\]

	Then equations \eqref{pf1} and \eqref{pf2} together imply
	\[
	\frac{d f^{[n]}}{dt} = \Delta_{\omega(t)} f^{[n]} + g(t) f^{[n]}.
	\]
	
	From Lemma \ref{vari}, we know that $\omega_{ij} \geq 0$. Given any $\tau > 0$, since each $\mathcal{D}^{[n]} = (V^{[n]}, E^{[n]}, F^{[n]})$ is a finite cellular decomposition, and by the estimate \eqref{keyesti2}, there exists a uniform constant $C_1$ such that $|\omega_{ij}(t)| \leq C_1$ for all $t \in [0, \tau]$. Hence, there exists a uniform constant $C$ such that
	\[
	\sum_{v_j \sim v_i} \omega_{ij}(t) \leq C
	\]
	for all $v_i \in V$ and all $t \in [0, \tau]$.
	
	Since $f^{[n]}(0) \geq 0$, the maximum principle (Lemma \ref{maximumprin}) implies $f^{[n]}(t) \geq 0$ for all $t \geq 0$, hence $K_i(u^{[n]}(t)) \leq 0$ for all $v_i \in V^{[n]}$. Passing to the limit, we obtain $K_i(u(t)) \leq 0$ for all $v_i \in V$, which implies, via \eqref{flow2}, that $u(t)$ is non-decreasing.
	
	Next, from Part II of the proof of Theorem \ref{flow_finite_exi}, we know that each $u_i(t)$ is uniformly bounded above. Therefore, the limit $\lim_{t \to \infty} u(t)$ exists.
	
	Finally, we prove that $K_i(\infty) = 0$. Since $K_i(u(t)) \leq 0$ for all $t$, we have $K_i(\infty) \leq 0$. Suppose, for contradiction, that $K_i(\infty) < 0$. Then there exists $T > 0$ such that $K_i(t) \leq K_i(\infty)/2 < 0$ for all $t \geq T$. It follows from \eqref{flow2} that
	\[
	\frac{du_i(t)}{dt} = -K_i(t) \geq -\frac{K_i(\infty)}{2} > 0
	\]
	for all $t \geq T$. Since the right-hand side is bounded away from zero, this implies $u_i(t) \to \infty$ as $t \to \infty$, contradicting the previously established upper bound. Therefore, $K_i(\infty) = 0$, completing the proof.
\end{proof}

\subsection{Convergence of the flow in Euclidean background geometry}

Let $\theta_{i j}$ denote half the inner angle at vertex $v_i$ contributed by the neighboring vertex $v_j$. Since $\Theta(e)=\frac{\pi}{4}$ for all edges $e \in E$, it follows that
$$
\theta_{i j}=\arctan \left(\frac{r_j}{r_i}\right)=\arctan \left(e^{u_j-u_i}\right),
$$
where $u_i = \log r_i$ and $u_j = \log r_j$.

Define $G(x):=\arctan \left(e^x\right)$. Then $G^{\prime}(0)=\frac{1}{2}$. Therefore, the curvature function $K_i(u)$ at vertex $v_i$ becomes
\begin{equation}\label{K_con}
	\begin{aligned}
		K_i(u) & =2 \pi-\sum_{j \sim i} 2 \theta_{i j} \\
		& =-\sum_{j \sim i}\left[2 G^{\prime}(0)\left(u_j-u_i\right)+2 G\left(u_j-u_i\right)-2 G^{\prime}(0)\left(u_j-u_i\right)-\frac{\pi}{2}\right] \\
		& =-\sum_{j \sim i}\left[\left(u_j-u_i\right)+F\left(u_j-u_i\right)\right]
	\end{aligned}
\end{equation}
where $F(x):=2 G(x)-x-\frac{\pi}{2}$. Clearly, $F(0)=0$ and $F^{\prime}(0)=0$.
\begin{lemma}\label{analysis_lemma}
	There exist universal constants $\delta_0, C_0>0$ such that for any $|x| \leq \delta_0$,
	$$
	|F(x)| \leq C_0|x|^2, \quad| F^\prime (x)| \leq C_0|x| .
	$$
\end{lemma}
\begin{proof}
	The Taylor expansion around $x=0$ shows that
	$$
	F(x)=\frac{1}{2} F^{\prime \prime}(0) x^2+o\left(x^2\right)
	$$
	which implies the stated bounds hold for sufficiently small $|x|$.
\end{proof}

Therefore, the flow equation \eqref{flow2} can be reformulated using \eqref{K_con} as the following semilinear parabolic equation:
\begin{equation}\label{semilinear_equation}
	\frac{du_i}{dt}= \Delta u_i + \tilde{F_i}(u),
\end{equation}
where $\tilde{F}_i(u):=\sum_{j \sim i} F\left(u_j-u_i\right)$.

To study equation \eqref{semilinear_equation}, we firstly study the the following homogeneous problem for the heat equation on $V$:
\begin{eqnarray}\label{homogenous}
	\left\{
	\begin{aligned}
		&\frac{d u}{dt}=\Delta u, \\
		&u(0)=\phi \in l^2(V) 
	\end{aligned}\right
	.
\end{eqnarray}
We have the following energy estimate:
\begin{lemma}\label{energy_est}
	Let $u$ be a solution to the homogeneous heat equation \eqref{homogenous}. Then for all $t \geq 0$,
		\begin{equation}\label{energy1}
		\begin{aligned}
			&N_{0,2}^2(u)+(1+t)N_{1,2}^2(u)+ \int_0^t N_{1,2}^2(u(\tau))d\tau+\int_0^t (1+\tau)N_{2,2}^2(u(\tau))d\tau \\
			&\leq (2+C_1) N_{0,2}^2(\phi),
		\end{aligned}
	\end{equation}
	where $C_1$ is the constant in \eqref{sob}.
\end{lemma}
\begin{proof}
	From \eqref{homogenous} and \eqref{stokes}, we know
	\begin{equation}\label{estimate1}
	\ddt{N_{0,2}^2(u)}=	\ddt{\|u\|_{l^2}^2}=2(u, \ddt{u})=2(u, \Delta u)= -N_{1,2}^2(u).
	\end{equation}
	From \eqref{commute} and \eqref{delta_def}, for all $D_i\in\mathcal{D}$, one has
	\begin{align}\label{homo1}
		\ddt{D_iu}=\Delta (D_i u)
	\end{align}
	For same reason in \eqref{estimate1}, from \eqref{homo1}, we have
	\begin{equation}\label{estimate2}
		\ddt{N_{0,2}^2(D_i u)}=		\ddt{\|D_i u\|_{l^2}^2}= - N_{1,2}^2(D_i u).
	\end{equation}
	Summing over all \(D_i \in \mathcal{D}\), and using \eqref{N1N0} and \eqref{N2N1}, yields
	\begin{equation}\label{estimate3}
		\ddt{N_{1,2}^2(u)}= - N_{2,2}^2(u).
	\end{equation}
	Hence 
	\begin{equation}\label{homo5}
		\ddt{(tN_{1,2}^2(u))}- N_{1,2}^2(u)+tN_{2,2}^2(u)=0.
	\end{equation}
	Combining \eqref{estimate1},\eqref{estimate3} and \eqref{homo5}, it follows
	\begin{equation}\label{homo6}
		\frac{d(2N_{0,2}^2(u)+(1+t)N_{1,2}^2(u))}{dt}+N_{1,2}^2(u)+(1+t)N_{2,2}^2(u)=0
	\end{equation}
	Integrate both sides of \eqref{homo6} over \([0,t]\), we can obtain the estimate
	\begin{equation}
		\begin{aligned}.
					&2N_{0,2}^2(u)+(1+t)N_{1,2}^2(u)-2N_{0,2}^2(\phi)-N_{1,2}^2(\phi)\\
					&+ \int_0^t N_{1,2}^2(u(\tau))d\tau+ \int_0^t (1+\tau)N_{2,2}^2(u(\tau))d\tau=0,
		\end{aligned}
	\end{equation}
	Hence, by \eqref{sob}, we have
	\begin{equation}
		\begin{aligned}
					&2N_{0,2}^2(u)+(1+t)N_{1,2}^2(u)+ \int_0^t N_{1,2}^2(u(\tau))d\tau+\int_0^t (1+\tau)N_{2,2}^2(u(\tau))d\tau \\
					&\leq (2+C_1) N_{0,2}^2(\phi)
		\end{aligned}
	\end{equation}
	which implies \eqref{energy1} naturally.
	
\end{proof}
Secondly, we study the the following non-homogeneous problem for the heat equation on $V$:
\begin{eqnarray}\label{nonhomogenous}
	\left\{
	\begin{aligned}
		&\ddt{u}=\Delta u+ \tilde{F}(h), \\
		&u(0)=0,
	\end{aligned}\right
	.
\end{eqnarray}
where $h$ is a given function in $C^1([0,\infty);l^2(H)).$

\begin{lemma}\label{energy_est2}
	Let $u$ be a solution to the non-homogeneous heat equation \eqref{nonhomogenous}. Then for all $t \geq 0$,
		\begin{equation}\label{energyineq2}
		\begin{aligned}
			&N_{0,2}^2(u)+(1+t)N_{1,2}^2(u)+ \int_0^t N_{1,2}^2(u(\tau))d\tau+\int_0^t (1+\tau)N_{2,2}^2(u(\tau))d\tau \\
			&\leq \int_0^t 4 \|\tilde{F}(h(\tau))\|_{l^2}\cdot N_{0,2}(u(\tau)) d\tau +\int_0^t 2(1+\tau)  \|\tilde{F}(h(\tau))\|_{l^2}\cdot  N_{2,2}(u(\tau))d\tau .
		\end{aligned}
	\end{equation}
\end{lemma}
\begin{proof}
	By similar reasoning as in Lemma \ref{energy_est}, we get
	\begin{equation}\label{nohomo1}
		\ddt{N_{0,2}^2(u)}+N_{1,2}^2(u)=2(\tilde{F}(h),u)\\
	\end{equation}
	and for all \(D_i \in \mathcal{D}\),
	\begin{equation}\label{781372}
		\ddt{N_{0,2}^2(D_iu)}+N_{1,2}^2(D_iu)=2(D_i \tilde{F}(h),D_i u),~~\forall D_i\in \mathcal{D}.
	\end{equation}
	From \eqref{stokes1}, we know
	\begin{equation}\label{12e132}
		\sum_{D_i\in \mathcal D}(D_i \tilde{F}(h),D_i u)=- 2( \tilde{F}(h), \Delta u)
	\end{equation}
	From \eqref{781372} and \eqref{12e132}, together with \eqref{N1N0} and \eqref{N2N1}, we find
	\begin{equation}\label{sajqdq}
		\ddt{N_{1,2}^2(u)}+N_{2,2}^2(u)=- 4( \tilde{F}(h), \Delta u)
	\end{equation}
	Thus, combining \ref{nohomo1} and \ref{sajqdq},                                                                                                                                                                                                                                                                                                                                                                                                                                                                                                                                                                                                                                                                                                                                                                                                                                                                                                                                                                                                                                                                                                                                                                            
	\begin{equation}\label{homo12}
		\frac{d(2N_{0,2}^2(u)+(1+t)N_{1,2}^2(u))}{dt}+N_{1,2}^2(u)+(1+t)(N_{2,2}^2(u)+4( \tilde{F}(h), \Delta u))=4(\tilde{F}(h),u)
	\end{equation}
	Integrate both sides of \eqref{homo12} over \([0,t]\), by \eqref{N02N22} and Cauchy-Schwarz inequality, we can obtain the estimate
	\begin{equation}\label{dqdqh}
		\begin{aligned}
					&2N_{0,2}^2(u)+(1+t)N_{1,2}^2(u)+ \int_0^t N_{1,2}^2(u(\tau))d\tau+\int_0^t (1+\tau)N_{2,2}^2(u(\tau))d\tau \\
					\leq & \int_0^t 4 |(\tilde{F}(h),u)| d\tau +\int_0^t 4(1+\tau)  |( \tilde{F}(h), \Delta u))| d\tau \\
					\leq &\int_0^t 4 \|\tilde{F}(h(\tau))\|_{l^2}\cdot N_{0,2}(u(\tau)) d\tau +\int_0^t 2(1+\tau)  \|\tilde{F}(h(\tau))\|_{l^2}\cdot  N_{2,2}(u(\tau))d\tau ,
		\end{aligned}
	\end{equation}
	which implies \eqref{energyineq2} naturally.
\end{proof}

\begin{lemma}\label{final_lemma_for}
	Let \(\{u_n\}_{n=1}^\infty\) be a sequence in \(l^2\) satisfying
	$$
	\left\|u_n\right\|_{l^2} \leq M, \quad \mathcal{E}\left(u_n\right):=\sum_{[i,j]\in E}|u_n(i)-u_n(j)|^2 \rightarrow 0, \quad n \rightarrow \infty
	$$
	Then
	$$
	\left\|u_n\right\|_{l^{\infty}} \rightarrow 0, \quad n \rightarrow \infty
	$$
\end{lemma}
\begin{proof}
	 We prove the lemma by contradiction. If it is not true, then there is a vertex $v_0\in V$  and a subsequence $\left\{u_{n_k}\right\}_{k=1}^{\infty}$ of sequence $\left\{u_n\right\}_{n=1}^{\infty}$ such that
	 $$
	 \left|u_{n_k}\left(v_0 \right)\right|>\epsilon, 
	 $$
	 where $\epsilon$ is a sufficient small positive constant. Since $\lim_{n\rightarrow +\infty} \mathcal{E}\left(u_n\right) \rightarrow 0$, we know for any $N> 0$, there is a $K =K(N)$ such that there are at least N vertices in $V$, which satisfies
	 $$
	 \left|u_{n_k}\left(v \right)\right|>\frac{\epsilon}{2}, \quad k \geq K.
	 $$
	 Hence,
	 $$
	 	\left\|u_{n_k}\right\|_{l^2}> \frac{\epsilon}{2} \sqrt{N}, \quad k \geq K.
	 $$ 
	 which contradicts with $\left\|u_n\right\|_{l^2} \leq M$, since $N$ can be arbitrary large.
\end{proof}

Theorem \ref{convergence2} is equivalent to the following.
\begin{theorem}
	There is a sufficient small positive constant $\epsilon$ such that, for any $\phi \in l^2(H)$ with $\|\phi\|_{l^2} \leq \epsilon$, there is a unique solution $u(t)$ to the equation \eqref{semilinear_equation} with initial value $u(0)=\phi$. Moreover,  $u(t)$ satisfies $u \in C^1([0, \infty), l^2(H))$,
	$$
	\|u(t)\|_{l^\infty} \rightarrow 0, \quad t\rightarrow \infty,
	$$
	and 
	$$
	\mathcal{E}(u(t)):=\sum_{[i,j]\in E}|u(i, t)-u(j, t)|^2 \leq C(1+t)^{-1},
	$$
	where $C= C(\epsilon)$ is a constant depending only on $\epsilon$.
\end{theorem}

\begin{proof}
	Define a norm
	\begin{equation}
		\begin{aligned}
			\|u\|_{T}&= \sup_{0\le t\le T} N_{0,2}(u(t))+\sup_{0\le t\le T}(1+t)^{\frac{1}{2}}N_{1,2}(u(t))\\
			&+ \left( \int_0^T N_{1,2}^2(u(\tau))d\tau   \right)^{\frac{1}{2}}+\left(\int_0^T (1+\tau)N_{2,2}^2(u(\tau))d\tau\right)^{\frac{1}{2}}
		\end{aligned}
	\end{equation}
	Define the function space
	\begin{align}\label{funciton_space}
		\mathcal{N}_{T,\delta}=\{u\in C^1([0,\infty);l^2(H)):u(0)=\phi,\|u\|_{T}\le\delta\}.
	\end{align}
	Let $u_I$ be the solution to the linear equation \eqref{homogenous}. By Lemma \ref{energy_est}, if $\epsilon\le \frac{\delta}{4\sqrt{2+C_1}}$ and $\|\phi\|_{l^2}=\|u(0)\|_{l^2} \leq \epsilon$, then $\|u_I\|_{T}\le \delta,~\forall T>0$, which means $$u_I\in \mathcal{N}_{T,\delta},~~\forall T>0.$$
	Hence $\mathcal{N}_{T,\delta}\neq\emptyset.$
	
	For any $h\in C^1([0,\infty);l^2(H)),$ we consider the the following non-homogeneous problem
	\begin{eqnarray}\label{contraction_map}
		\left\{
		\begin{aligned}
			&\ddt{u}=\Delta u+ \tilde{F}(h), \\
			&u(0)=\phi\in l^{2}(H),
		\end{aligned}\right
		.
	\end{eqnarray}
	where we require $\|\phi\|_{l^2}\leq \epsilon$, where $\epsilon$ is a positive constant, which will be determined later.
	
	By Corollary \ref{cor0}, the solution of \eqref{contraction_map} is unique. Let $\tilde{u}$ be the solution of \eqref{contraction_map} for $h$.
	One can define an operator $S(h):=\tilde{u}$.  
	
	Now, we are going to prove that $S$ is a contraction map in $\mathcal{N}_{T,\delta}$, when $\delta$ and $ \epsilon$ are sufficiently small.

	\emph{Step I: we prove that, for sufficient small $\delta$, we have $S(\mathcal{N}_{T,\delta})\subset\mathcal{N}_{T,\frac{\delta}{2}}$ for all $T>0.$}

	For any $h\in \mathcal{N}_{T,\delta}$, let $u_{II}$ be the solution to the equation \eqref{nonhomogenous}. 
	Then $\tilde{u}=u_I+u_{II}.$
	By Lemma \ref{energy_est} and \ref{energy_est2}, we have 
	\begin{equation}\label{cauchy1}
		\begin{aligned}
					\|\tilde{u}\|^2_{T}\le &C_2(\|\phi\|_{l^2}^2+\int_0^T 4 \|\tilde{F}(h(\tau))\|_{l^2}\cdot N_{0,2}(u_{II}(\tau)) d\tau \\
			&+\int_0^T 2 (1+\tau)  \|\tilde{F}(h(\tau))\|_{l^2}\cdot  N_{2,2}(u_{II}(\tau))d\tau ),
		\end{aligned}
	\end{equation}
	where constant $C_2$ depends only on $C_1$. 
	From Lemma \ref{analysis_lemma}, we know, when $\delta<\delta_0$, we have 
	\begin{equation}\label{Fhes}
		\|\tilde{F}(h(\tau))\|_{l^2} \leq C_3 N_{1,2}^2(h(\tau)),
	\end{equation}
	where $C_3$ only depends on $C_0$. Hence, 
	\begin{equation}\label{es1}
		\begin{aligned}
					&\int_0^T  \|\tilde{F}(h(\tau))\|_{l^2}\cdot N_{0,2}(u_{II}(\tau)) d\tau \\
					\leq & (\sup_{\tau\in[0,T]}\|\tilde{u}(\tau)\|_{l^2}+ \sup_{\tau\in[0,T]}\|{u_I}(\tau)\|_{l^2})\int_0^T C_3 N_{1,2}^2(h(\tau)) d\tau\\
					\leq &  (\sup_{\tau\in[0,T]}\|\tilde{u}(\tau)\|_{l^2}+(2+C_1) N_{0,2}^2(\phi))\int_0^T C_3 N_{1,2}^2(h(\tau)) d\tau\\
					\leq &   (\|\tilde{u}\|_{T}+(2+C_1) \epsilon^2)C_3 \delta^2
		\end{aligned}
	\end{equation}
	and
	\begin{equation}\label{es2}
		\begin{aligned}
				&	\int_0^T (1+\tau)  \|\tilde{F}(h(\tau))\|_{l^2}\cdot  N_{2,2}(u_{II}(\tau))d\tau\\
				\leq & \frac{1}{\delta_1}\int_0^T (1+\tau)  \|\tilde{F}(h(\tau))\|_{l^2}^2d\tau + \delta_1\int_0^T (1+\tau)  N_{2,2}^2(u_{II}(\tau))d\tau\\
				\leq & \frac{1}{\delta_1}\int_0^T (1+\tau)  C_3^2 N_{1,2}^4(h(\tau))d\tau + \delta_1\int_0^T (1+\tau)  (N_{2,2}(\tilde u(\tau))+N_{2,2}(u_{I}(\tau)) )^2d\tau\\
				\leq & \frac{1}{\delta_1}    C_3^2 \sup_{\tau\in[0,T]}((1+\tau)N_{1,2}^2(h(\tau)) )   \int_0^T  N_{1,2}^2(h(\tau))d\tau + 2\delta_1 \|\tilde{u}\|^2_{T} + 2\delta_1(2+C_1) \epsilon^2\\
				\leq & \frac{1}{\delta_1}    C_3^2 \delta^4 + 2\delta_1 \|\tilde{u}\|^2_{T} + 2\delta_1(2+C_1) \epsilon^2,
		\end{aligned}
	\end{equation}
	where $\delta_1 $ is a arbitrary positive constant.
	Combining \eqref{cauchy1}, \eqref{es1} and \eqref{es2}, we have
	\begin{equation}\label{esfinal}
		\begin{aligned}
			\|\tilde{u}\|^2_{T} &\leq C_2 (\epsilon^2+  4(\|\tilde{u}\|_{T}+(2+C_1) \epsilon^2)C_3 \delta^2+  \frac{2}{\delta_1}    C_3^2 \delta^4 + 4\delta_1 \|\tilde{u}\|^2_{T} + 4\delta_1(2+C_1) \epsilon^2)\\
			&\leq C_2 (  4\delta_1  \|\tilde{u}\|^2_{T} + 4C_3 \delta^2 \|\tilde{u}\|_{T} + 4(2+C_1) \epsilon^2C_3 \delta^2 + \epsilon^2+ \frac{2}{\delta_1}    C_3^2 \delta^4+ 4\delta_1(2+C_1) \epsilon^2  )\\
			&\leq C_2 (  4\delta_1  \|\tilde{u}\|^2_{T} + \delta_2 \|\tilde{u}\|^2_{T}+ \frac{16C^2_3 \delta^4}{\delta_2} + 4(2+C_1) \epsilon^2C_3 \delta^2 + \epsilon^2+ \frac{2}{\delta_1}    C_3^2 \delta^4+ 4\delta_1(2+C_1) \epsilon^2  )\\
			&\leq C_2 (  (4\delta_1 + \delta_2) \|\tilde{u}\|^2_{T}+( \frac{16C^2_3 }{\delta_2}+\frac{2C_3^2 }{\delta_1}  )\delta^4+ 4(2+C_1) \epsilon^2C_3 \delta^2 + (4\delta_1(2+C_1)+1) \epsilon^2  )\\
		\end{aligned}
	\end{equation}
	So, let $\delta_1= \frac{1}{10000 C_2}$, $\delta_2= \frac{1}{10000 C_2}$. Choose a positive constant 
	$$\delta <\frac{1}{100\sqrt{C_2( \frac{16C^2_3 }{\delta_2}+\frac{2C_3^2 }{\delta_1}  )}}$$
	 and a positive constant 
	 $$\epsilon< \min \{ \frac{1}{100\sqrt{(2+C_1)C_2C_3}}, \frac{\delta }{100 \sqrt{(4\delta_1(2+C_1)+1)C_2}}   \}.$$
	Then for any $h\in \mathcal{N}_{T,\delta}$, \eqref{esfinal} implies
	$$\| S(h)\|_{T}^2= \|\tilde{u}\|_{T}^2\le\frac{\delta^2}{4}.$$
	Thus $S(\mathcal{N}_{T,\delta})\subset\mathcal{N}_{T,\frac{\delta}{2}}$ for all $T>0.$
	
	\emph{Step II: we prove that, for sufficient small $\delta$, we have $S$ is a contraction map. }
	
	Now we prove that $S$ is a contraction map. Consider $h_1,h_2\in \mathcal{N}_{T,\delta}$. Let $w=S(h_1)-S(h_2)$, then $w$ satisfies the equation  
	\begin{equation}\label{difference}
		\left\{\begin{array}{l}
			\dfrac{dw}{dt}=\Delta w + \tilde{F}( h_1)-\tilde{F}( h_2), \\[5pt]
			w(0)=0.
		\end{array}\right.
	\end{equation}
	By \eqref{dqdqh}, we have 
	\begin{equation}\label{wes}
	\begin{aligned}
	\|w\|_T^2 & \leq C_4\left[\int_0^T\left|\Big(\tilde{F}\left(h_1\right)-\tilde{F}\left(h_2\right), w\Big)\right| \mathrm{d} \tau\right. \\
	& \left.+\int_0^T(1+\tau)\Big|\left(\tilde{F}\left(h_1\right)-\tilde{F}\left(h_2\right), \Delta w\right)\Big|d \tau \right]
\end{aligned}
	\end{equation}
	where $C_4$ is a universal constant. Since 
	\begin{equation}
		\begin{aligned}
			&|\tilde{F_i}( h_1)-\tilde{F_i}( h_2)|= \Big|\sum_{j\sim i}  \big[F(h_{1,j}-h_{1,i})-F(h_{2,j}-h_{2,i})\big]\Big|\\
			&=\Big| \sum_{j\sim i} \int_0^1 F^\prime(\sigma(h_{1,j}-h_{1,i})+(1-\sigma)(h_{2,j}-h_{2,i}))\cdot ((h_{1,j}-h_{1,i})-(h_{2,j}-h_{2,i}))d\sigma\Big|\\
			&\leq \sum_{j\sim i} C_0\int_0^1 \big|\sigma(h_{1,j}-h_{1,i})+(1-\sigma)(h_{2,j}-h_{2,i})|\cdot|(h_{1,j}-h_{1,i})-(h_{2,j}-h_{2,i})\big|d\sigma\\
			&\leq C_0\sum_{j\sim i}\big(\big|(h_{1,j}-h_{1,i})|+|(h_{2,j}-h_{2,i})|\big)\cdot|(h_{1,j}-h_{1,i})-(h_{2,j}-h_{2,i})\big|\\
		\end{aligned}
	\end{equation}
	Thus, 
		\begin{equation}\label{Fes}
		\begin{aligned}
			&\|\tilde{F}(h_1(\tau))-\tilde{F}(h_2(\tau))\|^2_{l^2}= \sum_i |\tilde{F_i}( h_1)-\tilde{F_i}( h_2)|^2\\
			&\leq \sum_i \big(C_0\sum_{j\sim i}  (|(h_{1,j}-h_{1,i})|+|(h_{2,j}-h_{2,i})|)\cdot|(h_{1,j}-h_{1,i})-(h_{2,j}-h_{2,i})|\big)^2\\
			&\leq 4C_0^2\sum_i \sum_{j\sim i}  (|(h_{1,j}-h_{1,i})|+|(h_{2,j}-h_{2,i})|)^2\cdot|(h_{1,j}-h_{1,i})-(h_{2,j}-h_{2,i})|^2\\
			&\leq 8C_0^2\sum_i \sum_{j\sim i}  (|(h_{1,j}-h_{1,i})|^2+|(h_{2,j}-h_{2,i})|^2)\cdot|(h_{1,j}-h_{2,j})-(h_{1,i}-h_{2,i})|^2\\
			 &\leq 8C_0^2 \big(N_{1,2}^2(h_1)+ N_{1,2}^2(h_2)\big)\cdot N_{1,2}^2(h_1-h_2)\\
		\end{aligned}
	\end{equation}
	So, by Hölder's inequality, we have
	\begin{equation}\label{Fes1}
		\begin{aligned}
		&	\int_0^T\left|\left(\tilde{F}\left(h_1\right)-\tilde{F}\left(h_2\right), w\right)\right| \mathrm{d} \tau \leq\int_0^T  \|\tilde{F}(h_1(\tau))-\tilde{F}(h_2(\tau))\|_{l^2} \cdot   \|w(\tau)\|_{l^2}       \mathrm{d} \tau\\
			&\leq \sqrt{8}C_0 (\sup_{\tau \in [0, T]} \|w(\tau)\|_{l^2} )\int_0^T   \big(N_{1,2}^2(h_1)+ N_{1,2}^2(h_2)\big)^{\frac{1}{2}}\cdot N_{1,2}(h_1-h_2)  \mathrm{d} \tau \\
			& \leq \sqrt{8}C_0 (\sup_{\tau \in [0, T]} \|w(\tau)\|_{l^2} )	\left( \int_0^T \big( N_{1,2}^2(h_1) + N_{1,2}^2(h_2) \big)\, \mathrm{d}\tau \right)^{1/2}
			\left( \int_0^T N_{1,2}^2(h_1 - h_2)\, \mathrm{d}\tau \right)^{1/2} \\
			& \leq 8C_0\delta^2 \left\| h_1 - h_2 \right\|_T^2
		\end{aligned}
	\end{equation}
	and
	\begin{equation}\label{Fes2}
		\begin{aligned}
			&	\int_0^T(1+\tau)\left|\left(\tilde{F}\left(h_1\right)-\tilde{F}\left(h_2\right), \Delta w\right)\right|d \tau  \\
	\leq &\int_0^T   (1+\tau)  \|\tilde{F}(h_1(\tau))-\tilde{F}(h_2(\tau))\|_{l^2} \cdot   \|\Delta w(\tau)\|_{l^2}       \mathrm{d} \tau\\
	\leq &\sqrt{8}C_0  \int_0^T   (1+\tau)   \big(N_{1,2}^2(h_1)+ N_{1,2}^2(h_2)\big)^{\frac{1}{2}}\cdot N_{1,2}(h_1-h_2)   \cdot   N_{2,2}(w)     \mathrm{d} \tau\\
	\leq &\sqrt{8}C_0  ( \sup_{\tau \in [0, T]} (1+\tau)^{\frac{1}{2}}   \big(N_{1,2}^2(h_1)+ N_{1,2}^2(h_2)\big)^{\frac{1}{2}} )\cdot \int_0^T   (1+\tau)^{\frac{1}{2}}  N_{1,2}(h_1-h_2)   \cdot   N_{2,2}(w)     \mathrm{d} \tau\\
	\leq &8C_0 \delta 	\left( \int_0^T (1+\tau) N_{2,2}^2(w)\, \mathrm{d}\tau \right)^{1/2}
	\left( \int_0^T N_{1,2}^2(h_1 - h_2)\, \mathrm{d}\tau \right)^{1/2}\\
	 \leq& 8C_0\delta^2 \left\| h_1 - h_2 \right\|_T^2.
		\end{aligned}
	\end{equation}
	Combining \eqref{wes}, \eqref{Fes1} and \eqref{Fes2}, 
\begin{equation}
	\|S(h_1)-S(h_2)\|_{T}^2 =\|w\|_{T}^2 \leq 16C_0C_4\delta^2 \left\| h_1 - h_2 \right\|_T^2
\end{equation}
	So $S$ is a contraction map when $16C_0C_4\delta^2< 1$.	Therefore, let $\delta$ be small enough, then there exists an $\epsilon$ which is also small enough, such that $S$ is a contraction map on $\mathcal{N}_{T,\delta}$ for any $T>0$.  
	
	\emph{Step III: we finish the proof. }
	
	By the contraction mapping theorem, $S$ has a unique fixed point $u$ in $\mathcal{N}_{T,\delta}$ such that $u$ is the solution of \eqref{semilinear_equation} with initial $u(0)=\phi\in l^{2}(H).$ From Theorem \ref{unique_for_uni_degree}, we know $u$ is unique in $C^1\left([0, \infty) ; l^2(H)\right)$.
	
	Since $u \in \mathcal{N}_{T,\delta}$, we know 
	$$
	\sup_{0\le t\le T}(1+t)^{\frac{1}{2}}N_{1,2}(u(t)) \leq  \| u\|_T \leq \delta,\quad  \forall T >0.
	$$
	and
	$$\sup_{0\le t\le T} N_{0,2}(u(t)) \leq  \| u\|_T \leq \delta,\quad  \forall T >0.$$
	Hence 
	$$
	\mathcal{E}(u(t)):=\sum_{[i,j]\in E}|u(i, t)-u(j, t)|^2 \leq C(1+t)^{-1},
	$$
	where $C$ is a constant. Therefore, by Lemma \ref{final_lemma_for}, we also know 
		$$
	\|u(t)\|_{l^\infty} \rightarrow 0, \quad t\rightarrow \infty.
	$$
\end{proof}

\subsection{Character in hyperbolic background geometry }
Firstly, we prove Theorem \ref{ch_hyper_exist}.
\begin{proof}

	By Theorem \ref{convergence1}, it suffices to construct a circle packing metric $u$ with $K_i(u)\leq 0$ for $i\in V$.
	
	Let $f(x)= \theta_{ij}(x\vec{1})$, from Lemma \ref{character_theta_dec}, we know $f(x)$ is differential and 
	\[
	\lim_{x\rightarrow 0} f(x) = \frac{\Theta_{ij}}{2}, \quad \lim_{x\rightarrow +\infty} f(x) = 0.
	\]
	Therefore, for any $\hat{c}>0$,there is a constant $\tilde C(\hat c)$ such that, if $x\leq\tilde C(\hat c)$, then
	$$f(x)\geq\frac{\Theta_{i j}}{2}-\hat c.$$
	Hence,  when $t \leq \tilde C(\hat c)$, we have
	\begin{equation}
		 \theta_{ij}(t \vec{1}) \geq \frac{\Theta_{i j}}{2}-\hat c
	\end{equation}
	Thus,
	\begin{equation}
		\begin{aligned}
			K_i(t \vec{1})  &= 2\pi - 2\sum_{j \sim i}\theta_{ij}(u^{[n]}(t))\\
			&\leq 2\pi - 2\sum_{j \sim i}\theta_{ij}(u_i^{[n]}(t)\vec{1})\\
			&\leq 2\pi - 2\sum_{j \sim i}( \frac{\Theta_{i j}}{2}-c)\\
			& =c d_j-(\mathcal{L}(\mathcal{D}, \Theta)_i-2 \pi) \\
			& \leq 0
		\end{aligned}
	\end{equation}
	
Choose $u = t\vec{1}$ with $t \leq \tilde{C}(\hat{c})$. Then $u$ is the desired circle packing metric.
	
\end{proof}

In following proof, we will use a lemma from Ge-Hua \cite{GH20}.
\begin{lemma}[\cite{GH20}, Lemma 3.9]\label{ghlemma}
	Let $f:[0,+\infty) \rightarrow \mathbb{R}$ be a locally Lipschitz function. Suppose that there is a constant $C_1$ such that $f^{\prime}(t) \leq 0$, for a.e. $t \in\left\{f>C_1\right\}$, then
	$$
	f(t) \leq \max \left\{f(0), C_1\right\}, \forall t \in[0,+\infty) .
	$$
	Similarly, if $f^{\prime}(t) \geq 0$, for a.e. $t \in\left\{f<C_1\right\}$, then
	$$
	f(t) \geq \min \left\{f(0), C_1\right\}, \forall t \in[0,+\infty)
	$$
\end{lemma}

	We denote $u_i^{[n]}(t)$ be the solution of \eqref{flow_finite}, Lemma \ref{flow_finite_exi}  guarantees the existence and uniqueness of $u_i^{[n]}(t)$ on $[0,\infty)$.

\begin{lemma}\label{hyper_lower_bound}
	If  the character $0<\hat{c}\leq\frac{\mathcal{L}_i(\mathcal{D}, \Theta)-2 \pi}{d_i}$ and $\hat{u_1}\leq u_i(0)$,  for any vertex $i \in V$ ($\hat c$ and -$\hat u$ are a positive constants), then there exists a constant $C_1(\hat{c},\hat{u_1} )<0$ such that
	$$
	C_1(\hat{c},\hat{u_1} ) \leq u_i^{[n]}(t) 
	$$0
\end{lemma}
\begin{proof}
	Since $V^{[n]}$ is finite, let 
	$$g(t):=\min _{i \in V^{[n]}} u_i^{[n]}(t).$$
	Then $g(t)$ is a locally Lipschitz function and for a.e. $t \in[0, T]$.
	
	If $ \hat{u_1} \leq g(t)$, from the flow equation \eqref{flow_finite} and $\hat{u_1}\leq u_i(0)$, it is obvious that for any $i \in V$, 
	$$
	\hat{u_1} \leq u_i^{[n]}(t).
	$$
	
	If $ g(t)\leq \hat{u_1}$, for $g(t)$,  there exists a special vertex $i \in V^{[n]}$ depending on $t$, such that
	$$
	g(t)=u_i^{[n]}(t), g^{\prime}(t)=u_i^{[n]^\prime}(t).
	$$
	Since $u_i^{[n]}(t)\leq u_j^{[n]}(t)$ for any $j \sim i$. From Lemma \ref{character_min_ineq}, we have
	\begin{equation}\label{ch1}
		\theta_{ij}(u^{[n]}(t)) \geq \theta_{ij}(u_i^{[n]}(t)\vec{1}).
	\end{equation}
	Let $f(x)= \theta_{ij}(x\vec{1})$, from Lemma \ref{character_theta_dec}, we know $f(x)$ is differential and $\lim_{x\rightarrow 0}= {\Theta_{i j}}/{2}$ and $\lim_{x\rightarrow +\infty}= 0$. Therefore, for any $\hat{c}>0$,there is a constant $\tilde C(\hat c)$ such that, if $x\leq\tilde C(\hat c)$, then
	$$f(x)\geq\frac{\Theta_{i j}}{2}-\hat c,$$
	Hence, when $g(t)\leq \tilde C(\hat c)$, we have
	\begin{equation}\label{ch2}
		\theta_{ij}(u_i^{[n]}(t)\vec{1})\geq \frac{\Theta_{i j}}{2}-\hat c
	\end{equation}
	By combining \eqref{ch1} and \eqref{ch2}, we have
	\begin{equation}
		\begin{aligned}
			K_i(u^{[n]}(t)) &= 2\pi - 2\sum_{j \sim i}\theta_{ij}(u^{[n]}(t))\\
			&\leq 2\pi - 2\sum_{j \sim i}\theta_{ij}(u_i^{[n]}(t)\vec{1})\\
			&\leq 2\pi - 2\sum_{j \sim i}( \frac{\Theta_{i j}}{2}-c)\\
			& =c d_j-(\mathcal{L}(\mathcal{D}, \Theta)_i-2 \pi) \\
			& \leq 0
		\end{aligned}
	\end{equation}
	By the flow equation \eqref{flow_finite} we know, if $g(t)\leq \tilde C(\hat c)$, we have  $g^{\prime}(t)=u_i^{[n]^{\prime}}(t)=-K_i(u^{[n]}(t))\geq 0$.  Then from Lemma \ref{ghlemma}, we know $g(t) \geq \min \{g(0), \hat{C}(\hat c)\}= \hat{C}(\hat c) , \forall t \in[0,+\infty)$.
	
	So, set $C_1(\hat{c},\hat{u_1} ) = \min \{\hat{C}(\hat c), \hat{u_1} \}$, by combining the case where $ \hat{u_1} \leq g(t)$ and $ g(t)\leq \hat{u_1}$, we have $  C_1(\hat{c},\hat{u_1} ) \leq u_i^{[n]}(t) $.
\end{proof}

\begin{lemma}\label{hyper_upper_bound}
 If the degree $d_i\leq \hat d$ and $u_i(0)\leq\hat u_2$, for any vertex $i \in V$ ($d$ and -$\hat u$ are positive constants). Then there exists a constant $C_2( \hat d, \hat u_2)<0$ such that
	$$
	 u_i^{[n]}(t) \leq C_2( \hat d, \hat u_2).
	$$
\end{lemma}
\begin{proof}
	From Lemma \ref{hyperconstr_u}, we know there exists a constant $\delta<0$ such that if $ \delta \leq u_i^{[n]}$, then 
	$$
	\theta_{ij}(u_i^{[n]}) < \frac{ \pi}{\hat d}, \quad \forall j\sim i.
	$$
	Thus, 
	$$
	\begin{aligned}
		K_i(u^{[n]}(t))&= 2\pi -2\sum_{j\sim i}\theta_{ij}(u_i^{[n]})\\
		&> 2\pi -2\sum_{j\sim i}\frac{ \pi}{\hat d}\\
		&\geq 2\pi -2\sum_{j\sim i}\frac{ \pi}{d_i}=0\\
	\end{aligned}
	$$
	For any $i \in V^{[n]}$, the flow equation \eqref{flow_finite} implies
	$$
	u_i^{[n]}(t) \leq \frac{1}{2} \max \{u_i(0), \delta\}\leq \frac{1}{2} \max \{\hat u_2, \delta\}
	$$
	For any $i \in V\setminus V^{[n]}$, we have
	$$u_i^{[n]}(t)= u_i(0)\leq\hat u_2$$
	Since $\delta$ is only depending on $\hat d$, therefore, set $C=\frac{1}{2} \max \{\hat u_2, \delta\}$ and we complete the proof.
\end{proof}
Theorem \ref{ch_hyper} is equivalent to
\begin{lemma}
	If for any $i \in V$, there are two positive constants $\hat{c}, \hat{d}$, and two negative constants $\hat u_1, \hat u_2$ such that
	\begin{itemize}
		\item [(1)] $\frac{\mathcal{L}_i(\mathcal{D}, \Theta)-2 \pi}{d_i}\geq \hat{c}.$
		\item [(2)] $d_i\leq \hat d.$
		\item [(3)]$\hat u_1 \leq u_i(0) \leq \hat u_2$
	\end{itemize}
then there is a unique solution $u(t)$ exist on $[0, +\infty)$, and exponentially approach to a zero discrete Gaussian curvature metric.
\end{lemma}
\begin{proof}
	Since $d_i\leq \hat d$, from Theorem \ref{unique_for_uni_degree}, we know there is a unique solution $u(t)$ exist on $[0, +\infty)$.
	
	Recall that $u_i^{[n]}(t)$ are the solution of equation \eqref{flow_finite}. Lemma \ref{existence_u} tell us that, $u(t)$ is a limit point of sequence $\{u^{[n]}(t)\}_{n=1}^{\infty}$. From Lemma \ref{hyper_lower_bound} and \ref{hyper_upper_bound}, we know, there are two negative constant $C_1(\hat{c},\hat{u_1} ), C_2( \hat d, \hat u_2)$ such that 
	\begin{equation}\label{u_bound}
		C_1(\hat{c},\hat{u_1} ) \leq u_i(t) \leq C_2( \hat d, \hat u_2), \forall i \in V.
	\end{equation}
	Since $K_i$ is a smooth function for $u$, we know there are constant $C_3(C_1,C_2)<0$ and $C_4(C_1,C_2)>0$ such that 
	\begin{equation}\label{K_bound}
		C_3(\hat{c},\hat{u_1} ) \leq K_i(t) \leq C_4( \hat d, \hat u_2), \forall i \in V.
	\end{equation}
	
	From Flow equation \ref{flow2}, we have
		\begin{equation}\label{keypf1}
		\begin{aligned}
			\frac{d}{d t} K_i(t)=&-\frac{\partial K_i}{\partial u_i} K_i(t)-\sum_{j \sim i} \frac{\partial K_i}{\partial u_j} K_j(t) \\
			=&-\sum_{j \sim i} \frac{\partial K_i}{\partial u_j} (K_j(t)-K_i(t))\\
			&-\Big(\frac{\partial K_i}{\partial u_i} +\sum_{j\sim i} \frac{\partial K_i}{\partial u_j}\Big)K_i(t).
		\end{aligned}
	\end{equation}
	Let $\omega_{ij}=-\dfrac{\partial K_i}{\partial u_j}$, $g_i=-\Big(\dfrac{\partial K_i}{\partial u_i} +\sum\limits_{j\sim i} \dfrac{\partial K_i}{\partial u_j}\Big)$, then equation \eqref{keypf1} is equivalent to 
	\begin{equation}\label{hyper_ch_basic}
		\frac{d K}{dt} = \Delta_{\omega} K + gK.
	\end{equation}
	From Lemma \ref{vari}, we know $\omega_{ij}=\omega_{ji}\geq 0$ and $g_i= -\dfrac{\partial (K_i+ \sum_{j\sim i} K_j)}{\partial u_i} <0$. Since $\dfrac{\partial K_i}{\partial u_i}$ and  $\dfrac{\partial K_i}{\partial u_j}$ are smooth functions for $u$, together with \eqref{u_bound}, we know, there are two constants $C_5(C_1, C_2)>0, \gamma= \gamma(C_1, C_2)>0$ such that
	\begin{equation}\label{omega_bound}
		\omega_{ij} \leq C_5(C_1, C_2), \forall [i,j] \in E,
	\end{equation}
	and
	$$g<-\gamma <0.$$ 
	
	Let $\tilde K(t) = e^{\gamma t}K(t)$, from \eqref{hyper_ch_basic}, we know 
		\begin{equation}\label{hyper_ch_basic2}
		\frac{d \tilde K}{dt} = \Delta_{\omega} \tilde K + (g+\gamma)\tilde K,
	\end{equation}
	where $\tilde K(0) = e^{\gamma 0}K(0)= K(0)$. So, from \eqref{K_bound}, we have
		\begin{equation}\label{tildeK0_bound}
		C_3(\hat{c},\hat{u_1} ) \leq \tilde K_i(0) \leq C_4( \hat d, \hat u_2), \forall i \in V.
	\end{equation}
	
	For any $T>0$, from \eqref{K_bound}, $\tilde K(t)$ is a bounded function in $V \times[0, T]$. From \eqref{omega_bound} and  $d_i\leq \hat d$, we have
		$$
	\sum_{j \sim i} \omega_{i j}(t) < C_5(C_1, C_2)\hat d.
	$$
	Thus, from Corollary \ref{equequi}, we have 
			\begin{equation}\label{tildeK_bound}
		C_3(\hat{c},\hat{u_1} ) \leq \tilde K_i(t) \leq C_4( \hat d, \hat u_2), \forall i \in V, t\in [0,T].
	\end{equation}
	Since $T$ is arbitrary, we know 
		\begin{equation}\label{tildeK_bound2}
		C_3(\hat{c},\hat{u_1} ) \leq \tilde K_i(t) \leq C_4( \hat d, \hat u_2), \forall i \in V, 
	\end{equation}
	which equals to
			\begin{equation}\label{K_bound2}
		C_3(\hat{c},\hat{u_1} )e^{-\gamma t} \leq  K_i(t) \leq C_4( \hat d, \hat u_2)e^{-\gamma t}, \forall i \in V.
	\end{equation}
	
\end{proof}

\subsection{Character in Euclidean background geometry }
In this section, we will give the proof of Theorem \ref{ch_euclid}.
We denote $u_i^{[n]}(t)$ be the solution of \eqref{flow_finite}, Lemma \ref{flow_finite_exi} tell us $u_i^{[n]}(t)$ exist on $[0,\infty)$ and is unique.

\begin{lemma}\label{euc_lower_bound}
	If  the character $2\pi \leq \mathcal{L}_i(\mathcal{D}, \Theta)$ and $\hat{u_1}\leq u_i(0)$,  for any vertex $i \in V$ ($\hat u_1$ is a negative constant), then there exists a constant $C_1(\hat{c},\hat{u_1} )<0$ such that
	$$
	\hat{u_1}  \leq u_i^{[n]}(t) 
	$$
\end{lemma}
\begin{proof}
	Since $V^{[n]}$ is finite, let 
	$$g(t):=\min _{i \in V^{[n]}} u_i^{[n]}(t).$$
	Then $g(t)$ is a locally Lipschitz function and for a.e. $t \in[0, T]$.
	
	If $ \hat{u_1} \leq g(t)$, from the flow equation \eqref{flow_finite} and $\hat{u_1}\leq u_i(0)$, it is obvious that for any $i \in V$, 
	$$
	\hat{u_1} \leq u_i^{[n]}(t).
	$$
	
	If $ g(t)\leq \hat{u_1}$, for $g(t)$,  there exists a special vertex $i \in V^{[n]}$ depending on $t$, such that
	$$
	g(t)=u_i^{[n]}(t), g^{\prime}(t)=u_i^{[n]^\prime}(t).
	$$
	Since $u_i^{[n]}(t)\leq u_j^{[n]}(t)$ for any $j \sim i$. From Lemma \ref{character_min_ineq}, we have
	\begin{equation}\label{ch_eu1}
		\theta_{ij}(u^{[n]}(t)) \geq \theta_{ij}(u_i^{[n]}(t)\vec{1}).
	\end{equation}
	From Lemma \ref{character_theta_dec}, we know 
	\begin{equation}\label{ch_eu2}
		\theta_{ij}(u_i^{[n]}(t)\vec{1})=\frac{\Theta_{i j}}{2}
	\end{equation}
	By combining \eqref{ch_eu1} and \eqref{ch_eu2}, we have
	\begin{equation}
		\begin{aligned}
			K_i(u^{[n]}(t)) &= 2\pi - 2\sum_{j \sim i}\theta_{ij}(u^{[n]}(t))\\
			&\leq 2\pi - 2\sum_{j \sim i}\theta_{ij}(u_i^{[n]}(t)\vec{1})\\
			&\leq 2\pi - 2\sum_{j \sim i}( \frac{\Theta_{i j}}{2})\\
			& =2\pi -\mathcal{L}_i(\mathcal{D}, \Theta)\\
			& \leq 0
		\end{aligned}
	\end{equation}
	By the flow equation \eqref{flow_finite} we know $g^{\prime}(t)=u_i^{[n]^{\prime}}(t)=-K_i(u^{[n]}(t))\geq 0$.  Hence,  $g(t) \geq g(0)\geq \hat{u_1} , \forall t \in[0,+\infty)$.
\end{proof}

\begin{lemma}\label{eucid_noninc}
	If $K_i(u(0))\geq 0$,  for any vertex, then 
	$$K_i(u^{[n]}(t))\geq 0$$
	and 
	$$
			u_i^{[n]}(t) \leq u_i(0).
	$$
\end{lemma}
\begin{proof}
	Recall that $K^{[n]}(t)=K(u^{[n]}(t))$, we denote 
	\begin{equation}
		f_i^{[n]}(t)= 
		\begin{cases}
			K_i^{[n]}(t)& \text{if } i \in V^{[n]}, \\[5pt]
			0 & \text{if }  i \notin V^{[n]}.
		\end{cases}
	\end{equation}
	Since $u^{[n]}(t)$ evolves according to the flow \eqref{flow_finite}, similar to the proof of Theorem \ref{convergence1}, we have 
	\begin{equation}
	\frac{df^{[n]}}{dt} = \Delta_{\omega(t)} f^{[n]}+g f^{[n]},
	\end{equation}
	where 
	\begin{equation}
		\omega_{ij}(t)= 
		\begin{cases}
			-\dfrac{\partial f_i^{[n]}}{\partial u_j} & \text{if } i,j \in V^{[n]}, \\[5pt]
			0 & \text{other cases.}
		\end{cases}
	\end{equation}
	and 
	\begin{equation}
		g(t) = 
		\begin{cases}
			-\Big(\dfrac{\partial f_i^{[n]}}{\partial u_i} +\sum\limits_{v_j \sim v_i, v_j \in V^{[n]}} \dfrac{\partial f_i^{[n]}}{\partial u_j}\Big) & \text{if } i \in V^{[n]}, \\[5pt]
			0 &\text{if } i \notin V^{[n]}.
		\end{cases}
	\end{equation}
	From Lemma \ref{vari}, we know $\omega_{ij}=\omega_{ij}\geq 0$. 
	
	For any $T>0$, since $\{\mathcal{D}^{[n]}=(V^{[n]}, E^{[n]}, F^{[n]})\}_{n=1}^{\infty}$ is  a finite cellular decomposition, $\frac{\partial K_i}{\partial u_i}$ and  $\frac{\partial K_i}{\partial u_j}$ are smooth functions for $u$, by \eqref{e2}, we know there is a uniform constant $C(\mathcal{D}^{[n]})$ such that 
	$$
	\sum_{v_j: v_j \sim v_i} \omega_{i j}(t) < C(\mathcal{D}^{[n]}),
	$$
	and 
	$$g(t)\leq C(\mathcal{D}^{[n]}),$$
	for any $v_i \in V,~t \in[0, T]$. 
	
	Since $f^{[n]}(0) \geq 0$, by Lemma \ref{maximumprin}, we know $f^{[n]}(t) \geq 0$, which implies $K_i(u^{[n]}(t))\geq 0$ for any $v_i \in V$. 
	$$
	u_i^{[n]}(t) \leq u_i^{[n]}(0).
	$$
\end{proof}

Theorem \ref{ch_euclid} is equivalent to

\begin{lemma}
	If for any $i \in V$, there is a negative constant $\hat u_1$ such that
	\begin{itemize}
		\item [(1)] $2\pi \leq \mathcal{L}_i(\mathcal{D}, \Theta);$
		\item [(2)]$\hat u_1 \leq u_i(0);$
		\item[(3)] $K_i(u(0))\geq 0.$
	\end{itemize}
	then there exists a solution $u(t)$ on $[0, +\infty)$, and  converges to a zero discrete Gaussian curvature metric.
\end{lemma}
\begin{proof}
	Let $u(t)$ be the solution which is obtained by a sequence $\{u^{[n]}(t)\}_{i=1}^{\infty}$ as shown in the proof of Theorem \ref{existence_u} with the initial value $u(0)$ such that $K(u(0)) \leq 0$.
	
	From Lemma \ref{eucid_noninc}, we know $K_i(u(t))\geq 0$. Then, by \eqref{flow2}, we know $u_i(t)$ is non-increasing. From Lemma \ref{euc_lower_bound}, $u(t)$ converges.
	
		In the end, we prove that $K_i(\infty)=0$. From $K_i(u(t))\geq 0$, we know $K_i(\infty)\geq 0$. By contradiction, if $K_i(\infty) >  0$, then there is a $T$ such that $K_i(t) \geq K_i(\infty)/2>0$ for any $t \geq T$. Thus from \eqref{flow2}, we have $du_i(t)/dt =- K_i(t)\leq - K_i(\infty)/2 <0$ for any $t \geq T$. Since $ K_i(\infty)/2$ is a constant, so $u_i(t)\rightarrow - \infty$ as $t\rightarrow \infty$, which contradicts to the fact that $u_i(t)$ has a lower bound.
\end{proof}

\section{Ideal hyperbolic polyhedra}
\label{ideal_polyhe}

In this section, we will study infinite ideal polyhedra in hyperbolic 3-space.

\subsection{Correspondence}\label{Correspondence}

An infinite (ideal) hyperbolic polyhedra is a (ideal) hyperbolic polyhedra with infinite faces in $\mathbb{H}^3$. 

Thurston \cite[Chapter 13]{Th76} had observed that there is a correspondence between circle patterns and hyperbolic polyhedra. For sphere $\mathbb{S}^2$, we observe the Poincar\'e ball model of three-dimensional hyperbolic space $\mathbb{H}^3$, where $\partial \mathbb{H}^3=\mathbb{S}^2$. Therefore, for any ideal circle pattern $\mathcal{P}=\left\{D_v\right\}_{v \in V}$ on $\mathbb{S}^2$, with the intersection angle function $\Theta$ taking values in $(0, \pi)$, we can associate each disk $D_v$ from $\mathcal{P}$ with a hyperbolic halfspace $H_v$, which is generated by the convex hull of the ideal set $\mathbb{S}^2 \backslash D_v \subset \partial \mathbb{H}^3$.  The polyhedron $P$ can be constructed by
$$
P=\bigcap_{v\in V} H_{v}.
$$
In the Poincar\'e ball model, the hyperbolic half-space $H_v$ corresponds to a Euclidean ball, so the dihedral angle between the faces $\partial H_{v_1}$ and $\partial H_{v_2}$ is equal to the intersection angle between the discs $D_{v_1}$ and $D_{v_2}$. Therefore, the polyhedron $(P, \Theta)$ corresponds with a circle pattern $(\mathcal P, \Theta)$.

Thus the question about the existence and uniqueness of a polyhedron with the given dihedral angle is equivalent to the existence and uniqueness of a circle pattern in the unit sphere with the given dihedral angle.

Since sphere $\mathbb{S}^2$ is compact, it can be seen that there are limit points of vertices for infinite circle pattern. If we give a picture for such infinite circle pattern, there are many infinitesimal small circles.  

\subsection{Ideal polyhedra in the infinite case}\label{ideal_poly_subsec}
\begin{figure}[h]
	\centering
	\includegraphics[width=0.88\textwidth]{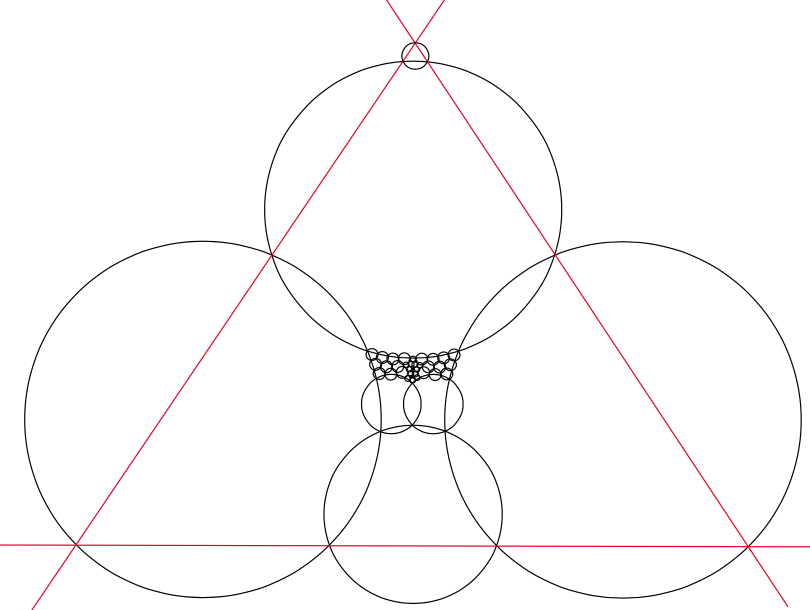}
	\caption{projected circle pattern in the plane}
	\label{projected_circle_pattern_in_sphere}
\end{figure}

For a circle pattern in the sphere, choose a 2-cell $f_{\infty} \in F$. Let $v_{f_\infty}$, which is the star vertex of $f_{\infty} $, as the center of the projection, we can project the circle pattern stereographically to the plane. In Figure \ref{projected_circle_pattern_in_sphere}, we give an example of the projected  circle pattern in the plane.

Let $V_\infty$ ($F_\infty$) be the set of vertices (faces) that incidence to $f_{\infty}$, $E_\infty$ be the set of edges that incidence to $V_\infty$.  The circle $C_v, v\in V_\infty$ is projected to the straight line and circle $C_v, v\in V \setminus V_\infty$ is projected to a circle. Since stereographic projection is conformal, the intersection angles are the same. Furthermore, a M\"obius-equivalent circle pattern in the sphere leads to a planar pattern which is similar, provided the same vertex is chosen as the center of projection.

The idea is to construct a circle pattern in the sphere by constructing the corresponding planar pattern using the Euclidean functional and then projecting it onto the sphere. So, the existence of an infinite circle pattern in the unit sphere is equivalent to the existence of the infinite circle pattern with finite straight lines in the plane. 

Denote $\mathcal{D^\prime}=(V \setminus V_\infty, E\setminus E_\infty, F\setminus F_\infty)$, which is an infinite cellular decomposition on $\mathbb{S}^2 \setminus F_\infty$. Let 
	\begin{equation}\label{K_euc}
	\hat	K_v= 
	\begin{cases}
		\hspace{6.5em}0, & \text{if }v \in  \operatorname{int}(V \setminus V_\infty), \\[5pt]
		\sum\limits_{v^\prime\in V_\infty,\; v\sim v^\prime }2 \Theta_{[v^\prime, v]},  & \text{if } v \in \operatorname{\partial}(V \setminus V_\infty).
	\end{cases}
	\end{equation}
By Thurston's construction, we can delete straight lines, which reforms this mathematical problem as the following problem:\\
\noindent{\bf Circle pattern problem in plane: }{\em In Euclidean background geometry, is there an ideal circle pattern with prescribed curvatures $\hat K$, which realizes $(\mathcal{D^\prime}, \Theta)$}?

Similar to flow \ref{flow}, we can introduce the following flow with prescribed curvatures $\hat K$:
\begin{equation}\label{flowwithprescribe}
	\frac{\mathrm{d} r_i}{\mathrm{~d} t}= -(K_i-\hat K_i) r_i, \quad \forall  v_i\in V
\end{equation}

By same reason in proving Theorem \ref{ch_euclid}, we can have the following theorem:

\begin{theorem}
	In Euclidean background geometry, let $\delta$ be any arbitrary small positive constant. If $\Theta \in (\delta, \pi-\delta)^{E}$ and satisfies (C1), then there exists a $\mathcal{D^\prime}$-type circle packing metric with zero discrete Gaussian curvature, if 
	\begin{itemize}
		\item [(1)] $2\pi \leq \mathcal{L}_i(\mathcal{D^\prime}, \Theta)+ \hat K_i$ for any vertex $i \in V$;
		\item[(2)] there exists a $\mathcal{D^\prime}$-type circle packing metric $\tilde r$ such that $K_i(\tilde r)\geq \hat K_i$.
	\end{itemize}
	Moreover, if there exists a positive constant $\hat r_1$ such that $\hat r_1 \leq r_i(0)$, $\forall i \in V$, then there exists a solution of $r(t)$ of equation \eqref{flow} with $r(0)= \tilde r$ such that $r(t)$ converges to a $\mathcal{D}$-type circle pattern $\tilde r ^{\prime}$ with prescribed curvatures $\hat K$.
\end{theorem}

Thus, we can have the theorem:

\begin{theorem}\label{infiniteICP}
	Let \( \mathcal{D} = (V, E, F) \)  be an infinite   cellular decomposition of the sphere. Suppose an angle $\Theta_e$ is assigned to each edge $e$. 
	
	Suppose $f_\infty\in F$ , and let  $V_\infty$ be the set of vertices that incident to $f_{\infty}$, $\hat K$ be defined as in \eqref{K_euc}.   Suppose the following condition hold.
	\begin{itemize}
		\item[(1)] $0<\Theta_e<\pi$ for any edge $e$.
		\item[(2)] $\sum_{e \in \gamma} \Theta_e = (|\gamma|-2) \pi$ if $\gamma$ is the boundary of a face.
		\item [(3)] $\sum_{j\sim i} \Theta_{ij}  + \hat K_i \geq 2 \pi$ for any vertex $i\in V \setminus V_\infty$.
		\item[(4)] there exist $\tilde r \in (0, + \infty)^{V \setminus V_\infty}$ such that $K_i(\tilde r)\geq \hat K_i$.
	\end{itemize}
	Then there exists an infinite ideal circle pattern in the sphere with intersection angles $\Theta_e$.
\end{theorem}
Therefore, we can get an infinite ideal circle pattern on the sphere. According to the correspondence described in Subsection \ref{Correspondence}, we can prove Theorem \ref{polyhedra_euc}.

\begin{remark}
     It is worth noting that those infinite ideal circle patterns may have self-intersections, see \cite{schramm}, which do not correspond to a real IIP. However, in the forthcoming paper \cite{GYZ}, we can prove that the self-intersections can be avoided.
\end{remark}

\subsection{Half Ideal polyhedra in the infinite case}\label{half_ideal_poly_subsec}
\begin{figure}[htbp]
	\centering
	\includegraphics[width=0.7\textwidth]{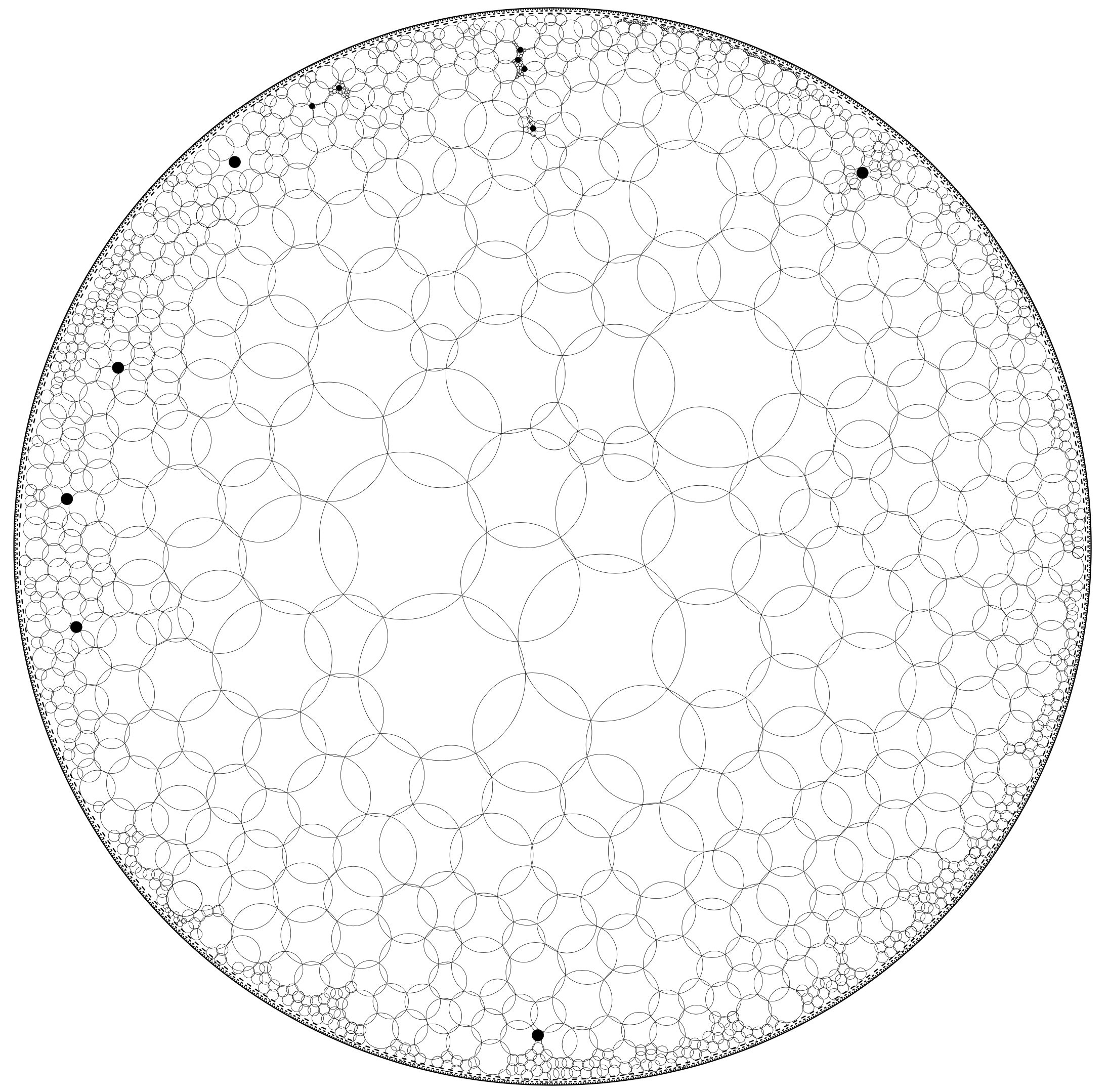}
	\caption{An example of circle pattern in a disc}
	\label{fig:infinite_circle_pattern_in_a_disk}
\end{figure}
In Subsection \ref{ideal_poly_subsec}, we are using the circle pattern in Euclidean background geometry to get the circle pattern in the sphere. In this section, we use the circle pattern in hyperbolic background geometry.

In infinite ideal polyhedra, there is a special type, which is 
\begin{definition}
	In the Poincaré ball model of three-dimensional hyperbolic space $\mathbb{H}^3$, where $\partial \mathbb{H}^3=\mathbb{S}^2$. An infinite ideal polyhedron $P=(V,E,F)$ is called \emph{half-infinite ideal polyhedra}, if 
	\begin{itemize}
		\item  There is one and only one face $f_\infty$ such that $\partial f_\infty$ are all the limit points  of $V$ in $\mathbb{S}^2$.
	\end{itemize}
\end{definition}

For a half-infinite ideal polyhedron $P=(V,E,F)$, assuming that $f_\infty$ passes through the center of the sphere $\mathbb{S}^2$, let $\mathcal{P}$ be the corresponding ideal circle pattern in the sphere for $P$. Choose any point $p \in f_\infty$ as the center of projection, we can project the circle pattern stereographically to the plane, which would give a circle pattern in a disk.

If we consider the dual infinite   cellular decomposition, according to the correspondence described in Subsection \ref{Correspondence}, from Theorem \ref{ch_hyper_exist}, we can prove Theorem \ref{polyhedra_hyp}.

It is worth noting that, in support of Theorem \ref{ch_hyper}, we also provide an algorithm for finding half-infinite ideal polyhedra with a fast convergence rate.

\bigskip
\textbf{Acknowledgements.} The authors express their gratitude to Bennett Chow, Feng Luo, Ze Zhou, and Lang Qin for their valuable discussions and insightful suggestions on combinatorial Ricci flows and circle packings. H. Ge is supported by NSFC, no.12341102, no.12122119. B. Hua is supported by NSFC, no.12371056, and by Shanghai Science and Technology Program [Project No. 22JC1400100]. 

\noindent Huabin Ge, hbge@ruc.edu.cn\\
\emph{School of Mathematics, Renmin University of China, Beijing 100872, P. R. China.}\\[-8pt]

\noindent Bobo Hua, bobohua@fudan.edu.cn\\[2pt]
\emph{School of Mathematical Sciences, LMNS, Fudan University, Shanghai, 200433, P.R. China.}\\[-8pt]

\noindent Hao Yu, b453@cnu.edu.cn\\
\emph{Academy for multidisciplinary studies, Capital Normal University, Beijing, 100048, P. R. China.}\\[-8pt]

\noindent Puchun Zhou, pczhou22@m.fudan.edu.cn\\
\emph{School of Mathematical Sciences, Fudan University, Shanghai, 200433, P.R. China.}
\end{document}